\documentclass[12pt]{article}

\usepackage[utf8]{inputenc}
\usepackage{tikz-cd}
\usepackage{amsmath}
\usepackage{amsthm}
\usepackage{amssymb}
\usepackage{graphicx}
\usepackage{enumitem}
\usepackage{verbatim}
\usepackage{caption}
\usepackage{mathrsfs}
\usepackage{geometry}
\usepackage{hyperref}
\usepackage{thmtools}
\usepackage{mathtools}

\newtheorem{manualtheoreminner}{Theorem}
\newenvironment{manualtheorem}[1]{%
	\renewcommand\themanualtheoreminner{#1}%
	\manualtheoreminner
}{\endmanualtheoreminner}

\newtheorem{theorem}{Theorem}[section]
\newtheorem{proposition}[theorem]{Proposition}
\newtheorem{proposition-definition}[theorem]{Proposition-Definition}
\newtheorem{lemma}[theorem]{Lemma}
\newtheorem{corollary}[theorem]{Corollary}

\theoremstyle{definition}
\newtheorem{definition}[theorem]{Definition}
\newtheorem{remark}[theorem]{Remark}
\newtheorem{question}[theorem]{Question}

\declaretheoremstyle[
numberlike=theorem,
bodyfont=\normalfont,
spaceabove=1em plus 0.75em minus 0.25em,
spacebelow=1em plus 0.75em minus 0.25em,
qed={$\diamondsuit$},
]{myExampleStyle}

\declaretheorem[
style=myExampleStyle,
title=Example,
refname={example,examples},
Refname={Example,Examples}
]{example}

\newcommand{\val}{\operatorname{val}}

\newcommand{\trop}{\operatorname{trop}}
\newcommand{\tropbar}{\overline{\operatorname{trop}}}
\newcommand{\rk}{\operatorname{rk}}
\newcommand{\supp}{\operatorname{supp}}

\newcommand{\sign}{\operatorname{sgn}}
\newcommand{\iin}{\operatorname{in}}
\newcommand{\pr}{\mathrm{pr}}

\newcommand{\Dr}{\operatorname{Dr}}
\newcommand{\Fl}{\operatorname{Fl}}
\newcommand{\FlDr}{\operatorname{FlDr}}
\newcommand{\LFl}{\operatorname{LFl}}
\newcommand{\LFlDr}{\operatorname{LFlDr}}

\newcommand{\defemph}[1]{\emph{#1}}

\title{Linear Degenerate Tropical Flag Varieties}
\author{Alessio Borz\`i, Victoria Schleis}

\begin{document}
	
	\maketitle
	
	\begin{abstract}
		We study linear degenerations of flag varieties from the point of view of tropical geometry. We define the \defemph{linear degenerate flag Dressian} and prove a correspondence between: $(a)$ points in the linear degenerate flag Dressian, $(b)$ linear degenerate flags of valuated matroids,  $(c)$ linear degenerate flags of tropical linear spaces, and $(d)$  morphisms of valuated matroids arising from projection maps.
		
		\par\vskip\baselineskip\noindent
		
	\end{abstract}
	
	\section{Introduction}
	
	The Grassmannian $G(r;n)$ parametrizes $r$-dimensional linear subspaces $U$ of an $n$-dimensional $K$-vector space $V$, and can be embedded in the projective space $\mathbb{P}^{\binom{n}{r}-1}$ with equations given by the Grassmann-Pl\"{u}cker relations.
	
	A generalization of the Grassmannian is the flag variety $\Fl(\mathbf{r};n)$, parametrizing flags of linear spaces of rank $\mathbf{r}=(r_1,\dots,r_k)$, that is, sequences of linear subspaces $(U_1,\dots,U_k)$ of $V$ with $\dim(U_i) = r_i$ such that $U_1\subseteq U_2 \subseteq \dots \subseteq U_k$. Flag varieties can be embedded in a product of projective spaces and the equations of this embedding are given by the incidence Pl\"{u}cker relations in addition to the Grassmann-Pl\"{u}cker relations.
	
	Flag varieties are further generalized by \defemph{linear degenerate flag varieties} parametrizing \defemph{linear degenerate flags} of linear subspaces. These are defined as follows: fix a sequence $\mathbf{f}$ of linear maps $f_i: V \rightarrow V$ for $i \in \{ 1,\dots,k-1 \}$ and a rank vector $\mathbf{r} = (r_1,\dots,r_k)$. An $\mathbf{f}$-linear degenerate flag is a sequence of linear subspaces $(U_1,\dots,U_k)$ of $V$ with $\dim(U_i) = r_i$ such that $f_i(U_i) \subseteq U_{i+1}$ for every $i \in \{ 1,\dots,k-1 \}$. The flag variety is the $\mathbf{f}$-linear degenerate flag  variety where all the $f_i$ are equal to the identity. Further, every $\mathbf{f}$-linear degenerate flag variety can be given as a sequence of projections, using the $GL(V)$ action on $V$ for each ${U}_i$ (see \cite[Lemma 2.6]{gabriel1972irreduciblequivers}). Hence, in this paper we will restrict to projections.
	
	The tropicalization $\tropbar(G(r;n))$ of the Grassmannian (considered inside the \emph{tropical projective space} $\mathbb{P}\left(\mathbb{T}^{\binom{n}{r}}\right)$ see Section \ref{sec: prelims tropgem} for more details) parametrizes \defemph{realizable} valuated matroids of rank $r$ on $n$ elements, or equivalently \defemph{realizable} tropical linear spaces (i.e. tropicalizations of linear spaces) of dimension $r$ in $\mathbb{P}(\mathbb{T}^n)$. The object parametrizing all valuated matroids of rank $r$ on $n$ elements, or equivalently all tropical linear spaces of dimension $r$ inside $\mathbb{P}(\mathbb{T}^n)$, is a tropical prevariety $\Dr(r,n)$ called the \emph{Dressian} (see \cite{speyer2008tropicallinearspaces}). Its equations are 
	the tropical
	Grassmann-Pl\"{u}cker relations, and we have $\tropbar(G(r;n)) \subseteq \Dr(r,n)$. In general, this  inclusion can be strict.
	
	Similarly, the tropicalization of the flag variety $\tropbar(\Fl(\mathbf{r};n))$ parametrizes realizable valuated flag matroids, or equivalently realizable flags of tropical linear spaces (see Section \ref{sec: prelims} for precise definitions).
	
	In \cite{haque2012tropical}, Haque defined the \emph{flag Dressian} $\FlDr(\mathbf{r};n)$, the tropical prevariety of the incidence Pl\"{u}cker relations and the Grassmann-Pl\"{u}cker relations. In \cite{brandt2021tropicalflag}, Brandt, Eur and Zhang proved that $\FlDr(\mathbf{r};n)$ parametrizes valuated flag matroids or equivalently flags of tropical linear spaces (see  Theorem \ref{thm: brandt equivalence for tropical flags} and \cite[Theorem 1]{haque2012tropical}).
	
	In this paper, we define the \emph{linear degenerate flag Dressian} $\LFlDr(\mathbf{r},\mathbf{S};n)$ of rank $\mathbf{r}=(r_1,\dots,r_k)$ and degeneration type $\mathbf{S}=(S_1,\dots,S_{k-1})$. Here the $r_i$ are the dimensions of the linear spaces and the $S_i$ define the projections $\pr_{S_i}$ by setting $\pr_{S_i}(e_j) = 0$ if $j \in S_i$ and $\pr_{S_i}(e_j) = e_j$ otherwise. The variety $\LFlDr(\mathbf{r},\mathbf{S};n)$ is defined as the tropical (pre)variety 
	of the linear degenerate incidence Pl\"{u}cker relations and the Grassmann-Pl\"{u}cker relations. The following is our main result, which is a generalization of the work of Brandt, Eur and Zhang \cite{brandt2021tropicalflag}.
	
	\begin{manualtheorem}{A}\label{thm:main_valuated_ld_flag_correspondence}
		Let $\boldsymbol{\mu}=(\mu_1,\dots,\mu_k)$ be a sequence of valuated matroids. The following statements are equivalent:
		\begin{enumerate}[label=(\alph*)]
			\item $\boldsymbol{\mu} \in \LFlDr(\mathbf{r},\mathbf{S};n)$;
			\item $\boldsymbol{\mu}$ is a linear degenerate valuated flag matroid (Definition \ref{def: ld flag matroid});
			\item $\pr_{S_i}^{\trop} \big( \tropbar(\mu_i) \big) \subseteq \tropbar(\mu_{i+1})$ for all $i \in \{ 1,\dots,k-1 \}$;
			\item every projection $\pr_{S_i}:\mu_{i+1}\rightarrow\mu_i$ is a morphism of valuated matroids (Definition \ref{def: morphism of matroid}).
		\end{enumerate}
	\end{manualtheorem}
	
	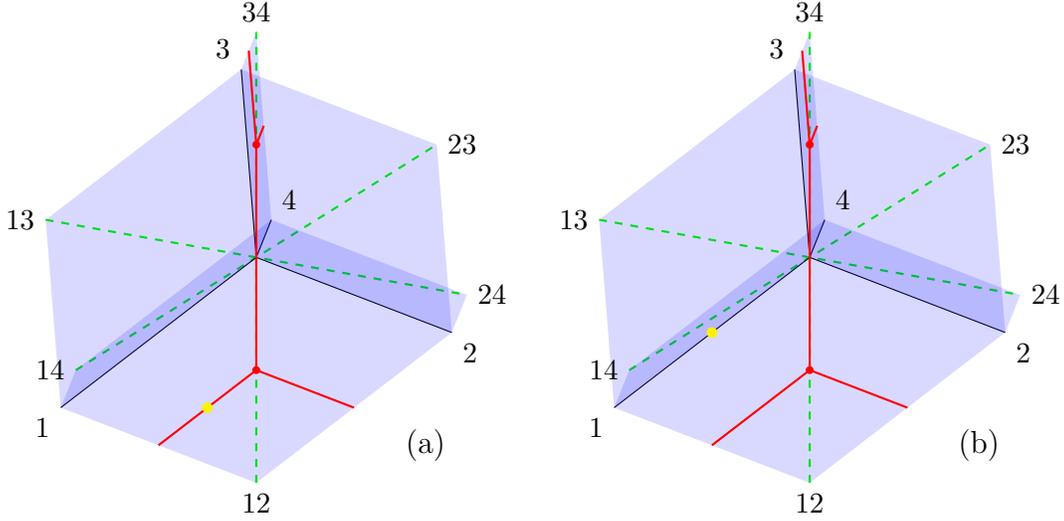
\begin{figure}[ht]
		\centering\begin{tabular}{c c c}
			\begin{tikzpicture}
				\draw[green, dashed,thick] (0,-3) -- (0,3)
				(-2.8,0.5)--(2.8,-0.5)
				(-2.4, -1.5) -- (2.4,1.5);
				\draw (0,0)--(-2.6,-2)
				(0,0) -- (-0.2, 2.5)
				(0,0) -- (0.2, 0.5)
				(0,0) -- (2.6,-1);
				
				\fill[opacity = 0.15, blue](0,0) -- (-0.2,2.5) --  	(-2.8,0.5)--(-2.6,-2)--cycle;
				\fill[opacity = 0.15, blue](0,0) -- (-0.2, 2.5) -- (0,3) -- (0.2, 0.5)	--cycle;
				\fill[opacity = 0.15, blue](0,0)  -- (0.2, 0.5) -- (2.8,-0.5) -- (2.6,-1)	--cycle;
				\fill[opacity = 0.15, blue](0,0)  --  (-0.2, 2.5)-- (2.4,1.5) -- (2.6,-1)	--cycle;
				\fill[opacity = 0.15, blue](0,-3)  --  (-2.6,-2)--(0,0)   -- (2.6,-1)	--cycle;
				\fill[opacity = 0.15, blue] (0.2, 0.5)	--(-2.4, -1.5) --(-2.6,-2)--(0,0)  	--cycle;
				\node[below left] at (-2.6,-2) {\small$1$};
				\node[above left] at (-0.2, 2.5) {\small$3$};
				\node[ above right] at (0.2, 0.5){\small$4$};
				\node[ below right] at (2.6,-1){\small$2$};
				\node[ below ] at (0,-3) {\small{$12$}};
				\node[ above ] at (0,3) {\small{$34$}};
				\node[ left ] at(-2.4, -1.5) {\small{$14$}};
				\node[ right ] at(2.4, 1.5) {\small{$23$}};
				\node[ left ] at	(-2.8,0.5) {\small{$13$}};
				\node[ right ] at(2.8,-0.5){\small{$24$}};
				\draw[red, thick] (-1.3,-2.5) -- (0,-1.5) -- (1.3,-2)  (0,-1.5) -- (0,0) -- (0, 1.5) -- (-0.1, 2.75) (0,1.5) -- (0.1,1.75);
				\fill[red] (0, -1.5) circle (1.5pt);
				\fill[red]  (0, 1.5)  circle (1.5pt);
				
				\fill[yellow] (-0.65,-2) circle (2pt);
				\node at  (2.25,-2.5){(a)};
			\end{tikzpicture} &\begin{tikzpicture}
				\draw[green,dashed, thick] (0,-3) -- (0,3)
				(-2.8,0.5)--(2.8,-0.5)
				(-2.4, -1.5) -- (2.4,1.5);
				\draw (0,0)--(-2.6,-2)
				(0,0) -- (-0.2, 2.5)
				(0,0) -- (0.2, 0.5)
				(0,0) -- (2.6,-1);
				
				\fill[opacity = 0.15, blue](0,0) -- (-0.2,2.5) --  	(-2.8,0.5)--(-2.6,-2)--cycle;
				\fill[opacity = 0.15, blue](0,0) -- (-0.2, 2.5) -- (0,3) -- (0.2, 0.5)	--cycle;
				\fill[opacity = 0.15, blue](0,0)  -- (0.2, 0.5) -- (2.8,-0.5) -- (2.6,-1)	--cycle;
				\fill[opacity = 0.15, blue](0,0)  --  (-0.2, 2.5)-- (2.4,1.5) -- (2.6,-1)	--cycle;
				\fill[opacity = 0.15, blue](0,-3)  --  (-2.6,-2)--(0,0)   -- (2.6,-1)	--cycle;
				\fill[opacity = 0.15, blue] (0.2, 0.5)	--(-2.4, -1.5) --(-2.6,-2)--(0,0)  	--cycle;
				\node[below left] at (-2.6,-2) {\small$1$};
				\node[above left] at (-0.2, 2.5) {\small$3$};
				\node[ above right] at (0.2, 0.5){\small$4$};
				\node[ below right] at (2.6,-1){\small$2$};
				\node[ below ] at (0,-3) {\small{$12$}};
				\node[ above ] at (0,3) {\small{$34$}};
				\node[ left ] at(-2.4, -1.5) {\small{$14$}};
				\node[ right ] at(2.4, 1.5) {\small{$23$}};
				\node[ left ] at	(-2.8,0.5) {\small{$13$}};
				\node[ right ] at(2.8,-0.5){\small{$24$}};
				\draw[red, thick] (-1.3,-2.5) -- (0,-1.5) -- (1.3,-2)  (0,-1.5) -- (0,0) -- (0, 1.5) -- (-0.1, 2.75) (0,1.5) -- (0.1,1.75);
				\fill[red] (0, -1.5) circle (1.5pt);
				\fill[red]  (0, 1.5)  circle (1.5pt);
				
				\fill[yellow] (-1.3,-1) circle (2pt);
				\node at  (2.25,-2.5){(b)};
			\end{tikzpicture}

		\end{tabular}
		\caption{{(a):} A tropical flag in $\trop \big( \Fl((1,2,3);4) \big)$. {(b):} A tropical linear degenerate flag in $\trop \big( \LFl \big( (1,2,3),(\{1\},\emptyset);4 \big) \big)$. Both are made of a yellow point, a red tropical line, and a blue tropical plane. The additional subdivision given by the green dashed rays on the tropical plane is useful for describing the (linear degenerate) tropical flag varieties, see Examples \ref{ex: tropfl4} and \ref{ex: lindeg 1 fl4}.}
		\label{fig: tropical flag}
	\end{figure}
	
	An analogous statement holds for the realizable case. We obtain the following correspondence:
	
	\begin{manualtheorem}{B}\label{thm:main_valuated_ld_flag_correspondence_realizable}
		Let $\boldsymbol{\mu}=(\mu_1,\dots,\mu_k)$ be a sequence of \emph{realizable} valuated matroids. The following statements are equivalent:
		\begin{enumerate}[label=(\alph*)]
			\item $\boldsymbol{\mu} \in \tropbar(\LFl(\mathbf{r},\mathbf{S};n))$;
			\item $\boldsymbol{\mu}$ is a \emph{realizable} linear degenerate valuated flag matroid (Definition \ref{def: ld flag matroid});
			\item there exist realizations $L_1, \dots, L_{k}$ of $\mu_1,\dots,\mu_k$ such that $\pr_{S_i}(L_i)\subseteq L_{i+1}$ for all  $i \in \{1,\dots,k-1\}$;
			\item every projection $\pr_{S_i}:\mu_{i+1}\rightarrow\mu_i$ is a \emph{realizable} morphism of valuated matroids (see Definition \ref{def: morphism of matroid}), and all valuated matroids and morphisms can be realized \emph{simultaneously}.
		\end{enumerate}
	\end{manualtheorem}
	
	The two extreme cases of linear degenerate flag varieties are flag varieties (by setting all $S_i = \emptyset$) and products of Grassmannians (by setting all $S_i = \{1,\dots,n\}$). Hence, linear degenerate flag varieties bridge the gap between products of Grassmannians and flag varieties, see Section \ref{sec: applications} for more details.

	Linear degenerate flag varieties are of interest in representation theory. They were studied in \cite{fourier2017lineardegenerate, fourier2020lineardegenerations, feigin2011genocchi,feigin2010grassmanndegenerations, feigin2013frobeniussplittin, putz2022degenerate}. It was observed in \cite{cerulliirelli2012quiveranddegenerate} that linear degenerate flag varieties are isomorphic to quiver Grassmannians for type A quivers. Quiver Grassmannians first appeared in \cite{crawle-boevey1989quiver, schofield1992quiver}. They are varieties parametrizing subrepresentations of quiver representations. Notably, every projective variety is isomorphic to a quiver Grassmannian \cite{reineke2013projectiveisquiver}. In addition, every quiver Grassmannian can be naturally embedded in a product of Grassmannians. The equations of this embedding were described in \cite{lorscheid2019pluckerelations}.
	
	A first step towards \defemph{tropical flag varieties} was made in \cite{haque2012tropical} by Haque, showing that the flag Dressian parametrizes flags of tropical linear spaces. The tropicalization of (complete) flag varieties of length $4$ and $5$ were computed in \cite{bossinger2017toricdegenerations}. In a related work \cite{fourier2019tropicalflagvarieties}, tropical flag varieties are linked to PBW-degenerations. Further, in \cite{jell2022moduli}, in order to combinatorially describe tropicalized Fano schemes, previously introduced in \cite{lamboglia2022fano}, the authors studied the space of valuated flag matroids $(\mu_1,\mu_2)$ of rank $(r,r+1)$ where $\mu_2$ is fixed. A generalization of valuated flag matroids to tracts has been made in \cite{jarra2022flag}. Positivity for tropical flag Dressians has been studied in \cite{arkani2021positive,boretsky2022totally,boretsky2022polyhedral,joswig2021subdivisions,speyer2021positive}. 
	
	This paper is structured as follows. In Section \ref{sec: prelims} we review the needed background knowledge and fix our notation. In Section \ref{sec: linear degenerate section} we define the linear degenerate flag Dressian and prove our main results, Theorem \ref{thm:main_valuated_ld_flag_correspondence} and Theorem \ref{thm:main_valuated_ld_flag_correspondence_realizable}. Finally, in Section \ref{sec: applications}  we give some examples, first results on relationships of linear degenerate flag varieties and Dressians with similar rank but different degeneration sets, and discuss some possible applications and future work.
	
	\paragraph{Acknowledgements}
	We thank Hannah Markwig and  Ghislain Fourier for suggesting the topic and for helpful conversations. 
	We further thank Diane Maclagan for helpful conversations and useful insights. Part of the work on this project was done while the first author visited the University of T\"{u}bingen in the Fall of 2021, he would like to thank Hannah Markwig and the second author for the invitation and the institute staff for their hospitality and support. The second author further thanks Ghislain Fourier for the invitation to the RWTH Aachen in the summer of 2022 and his hospitality throughout the visit. The first author was funded through the Warwick Mathematics Institute Centre for Doctoral Training, and gratefully acknowledges the support of the University of Warwick. The second author was supported by the Deutsche Forschungsgemeinschaft (DFG, German Research
	Foundation), Project-ID 286237555, TRR 195.
	
	\section{Preliminaries}\label{sec: prelims}
	
	\subsection{Tropical Geometry}\label{sec: prelims tropgem}
	
	In this section, we review the basics of tropical geometry. Our main reference is \cite{maclagan2015book}. We follow the $\min$-convention for all tropical and matroidal operations.
	
	Set $\mathbb{T} = \mathbb{R} \cup \{ \infty \}$ and define $a \oplus b = \min\{a,b\}$ and $a \odot b = a+b$ for every $a,b \in \mathbb{T}$. Then $(\mathbb{T},\oplus,\odot)$ is a semifield, called the \defemph{tropical semifield}. The \defemph{tropical projective space} is $\mathbb{P}(\mathbb{T}^{n}) = (\mathbb{T}^{n} \setminus \{ (\infty,\dots,\infty) \}) / \mathbb{R} \mathbf{1} = (\mathbb{T}^{n} \setminus \{ (\infty,\dots,\infty) \}) / \sim$. Here we are quotienting by the equivalence relation $\mathbf{u} \sim \mathbf{v}$ if $\mathbf{u} = \mathbf{v} + c \mathbf{1}$ for some $c \in \mathbb{R}$, where $\mathbf{1}=(1,\dots,1) \in \mathbb{R}^{n}$. A \defemph{tropical polynomial} is an element of the semiring $\mathbb{T}[x_{1},\dots,x_n]$ in the variables $x_{1},\dots,x_n$ with coefficients in $\mathbb{T}$. 
	The \defemph{tropical hypersurface} of a tropical polynomial $F = \bigoplus_{u \in \mathbb{N}^{n}} c_u \odot x^u \in \mathbb{T}[x_{1},\dots,x_n]$ is
	\[ V(F) = \Bigg\{ x \in \mathbb{P}(\mathbb{T}^n) : \min_{ u \in \mathbb{N}^{n}}\left\{ c_u + \sum_{i=1}^n u_i \cdot x_i \right\}\text{is achieved at least twice} \Bigg\}, \]
	where whenever $\min_{u\in \mathbb{N}^{n}}\{ c_u + \sum_{i=1}^n u_i \cdot x_i \} = \infty$, by convention the minimum is achieved at least twice, even if the expression is a tropical monomial. The \defemph{tropical variety} of an ideal of tropical polynomials $J \subseteq \mathbb{T}[x_{1},\dots,x_n]$ is defined by
	\[ V(J) = \bigcap_{F \in J} V(F) \subseteq \mathbb{P}(\mathbb{T}^n). \]
	In the following, let $K$ be a field with valuation $\val:K \rightarrow \mathbb{T}$. The \defemph{tropicalization} of a polynomial $f = \sum_{u \in \mathbb{N}^{{n}}} a_u x^u \in K[x_{1},\dots,x_n]$ is the tropical polynomial
	\[ \trop(f) = \bigoplus_{u \in \mathbb{N}^{{n}}} \val(a_u) \odot x^u \in \mathbb{T}[x_{1},\dots,x_n]. \]
	
	The \defemph{tropicalization} $\trop(I)$ of an ideal $I \subseteq K[x_1,\dots,x_n]$ is the ideal of tropical polynomials generated by the tropicalizations of all polynomials in $I$:
	\[ \trop(I) = \langle \trop(f) : f \in I \rangle \subseteq \mathbb{T}[x_1,\dots,x_n]. \]
	Over an algebraically closed base field $K$ with a non-trivial valuation, the \defemph{tropicalization} $\tropbar(X)$ of a subvariety $X \subseteq \mathbb{P}_K^n$, is defined by
	\[ \tropbar(X) = \overline{ \{ (\val(x_0),\dots,\val(x_n)) \in \mathbb{P}(\mathbb{T}^{n+1}) : [x_0:\dots:x_n] \in X \}}, \]
	where the closure is with respect to the Euclidean topology induced on $\mathbb{P}(\mathbb{T}^{n+1})$.
	
	Now we have two possible ways of constructing a tropical variety from a homogeneous ideal $I \subseteq K[x_0,\dots,x_n]$: we can first tropicalize the ideal, and then take its tropical variety $V(\trop(I))$, or we can consider the affine variety $V(I)$ and tropicalize it to obtain $\tropbar(V(I))$. The Fundamental Theorem of Tropical Geometry assures us that, over algebraically closed fields with nontrivial valuation, these two operations yield the same result, i.e. that $ \tropbar(V(I)) = V(\trop(I))$, see \cite[Theorem 6.2.15]{maclagan2015book}.
	
	We use the notation $\tropbar(X)$ for the tropicalization of a subvariety $X \subseteq \mathbb{P}^n$ and denote by $\trop(X)$ the intersection $\tropbar(X) \cap (\mathbb{R}^{n+1} / \mathbb{R} \mathbf{1})$, which is the tropicalization inside $\mathbb{R}^{n+1} / \mathbb{R} \mathbf{1}$ of the intersection $X \cap T^n$, where $T^n \simeq (K \setminus \{ 0 \})^n$ is the algebraic torus inside $\mathbb{P}^n$. That is, $\tropbar(X)$ might contain points in which some of the coordinates are $\infty$. For a thorough description of this extension of tropical varieties, we refer to \cite[Section 2]{brandt2021tropicalflag} and \cite[Section 6.2]{maclagan2015book}.
	
	Further, tropical varieties have a nice polyhedral structure: for an irreducible subvariety $V(I)$ of dimension $d$ in the torus $T^n$, the tropical variety $\trop(V(I))$ is the support of a $d$-dimensional, rational, balanced polyhedral complex that is connected through codimension one, see \cite[Theorem 3.3.5]{maclagan2015book} for more details.

	\subsection{Valuated matroids and tropical linear spaces}\label{sec: prelims valuated matroids}
	
	Throughout, we will assume that the reader is familiar with basic matroid theory, see  \cite{oxley2011} and \cite{welsh1976mtheory}. We write $[n]$ for the set $\{1,2,\dots, n\}$ and $\binom{[n]}{r}$ for the family of subsets of $[n]$ of cardinality $r$.
	
	\begin{definition}\label{def:valuated matroid}
		A \defemph{valuated matroid} of rank $r$ on the ground set $[n]$ is a function $\nu: \binom{[n]}{r} \rightarrow \mathbb{T}$ such that $\nu(B) \neq \infty$ for some $B \in \binom{[n]}{r}$, and for all $I,J \in \binom{[n]}{r}$ and $i \in I\setminus J$ there exists $j \in J \setminus I$ satisfying
		\[ \nu(I)+\nu(J) \geq \nu((I \setminus i) \cup j) + \nu((J \setminus j) \cup i). \]
	\end{definition}
	
	
	Two rank $r$ valuated matroids $\mu$ and $\nu$ on a common ground set $[n]$ are \emph{equivalent} if there exists $a \in \mathbb{R}$ such that $\mu(B) = \nu(B) + a$ for every $B \in \binom{[n]}{r}$. In other words, every equivalence class of a valuated matroid $\nu: \binom{[n]}{r} \rightarrow \mathbb{T}$ can be seen as a point in  $\mathbb{P}\left(\mathbb{T}^{\binom{n}{r}}\right)$. Throughout, we will regard two equivalent valuated matroids as being the same, and consider valuated matroids only up to equivalence.
	
	If $\nu: \binom{[n]}{r} \rightarrow \mathbb{T}$ is a valuated matroid, then $\{ B \in \binom{[n]}{r} : \nu(B) \neq \infty \}$ is a collection of bases of a matroid $N$, called the \defemph{underlying matroid} of $\nu$.
	
	\begin{definition}[{\cite[Proposition 7.4.7]{brylawski1986matroidquotients}}]
		Let $M$ and $N$ be two matroids over the same ground set $[n]$. We say that $M$ is a \defemph{matroid quotient} of $N$,  denoted $M \twoheadleftarrow N$, if every flat of $M$ is a flat of $N$.
	\end{definition}

	\begin{definition}[{\cite[Definition 4.2.2]{brandt2021tropicalflag}}]
		Let $\mu$ and $\nu$ be two valuated matroids on the ground set $[n]$ of rank $r \leq s$ respectively. We say that $\mu$ is a \defemph{valuated matroid quotient} of $\nu$, denoted $\mu \twoheadleftarrow \nu$, if for every $I \in \binom{[n]}{r}$, $J \in \binom{[n]}{s}$ and $i \in I \setminus J$, there exists $j \in J \setminus I$ such that
		\[ \mu(I) + \nu(J) \geq \mu(I \cup j \setminus i) + \nu(J \cup i \setminus j). \]
	\end{definition}
	
	If $\mu \twoheadleftarrow \nu$ is a valuated matroid quotient, and $M$ and $N$ are the underlying matroids of $\mu$ and $\nu$ respectively, then $M \twoheadleftarrow N$.
	
	\begin{definition}\label{def: valuated flag matroid}
		A sequence of valuated matroids $\boldsymbol{\mu} = (\mu_1,\dots,\mu_k)$ on a common ground set $[n]$ is a \defemph{valuated flag matroid} if $\mu_i \twoheadleftarrow \mu_j$ for every $1 \leq i \leq j \leq k$. Analogously, a sequence of matroids $\mathbf{M} = (M_1,\dots,M_k)$ on $[n]$ is a \defemph{flag matroid} if $M_i \twoheadleftarrow M_j$ for every $1 \leq i \leq j \leq n$.
	\end{definition}
	Let $K$ be a field 
	with valuation $\val:K \rightarrow \mathbb{T}$. Let $L$ be an $r$-dimensional vector subspace of $K^n$ given as the row span of a matrix $A$. We denote by $(p_I)$ the \emph{Pl\"ucker coordinates} of $A$, where $p_I$ is defined as the minor of $A$ indexed by $I\in\binom{[n]}{r}$. Then, the function $\mu(A): \binom{[n]}{r} \rightarrow \mathbb{T}$ defined by $I \mapsto \val(p_I)$ is a valuated matroid satisfying Definition \ref{def:valuated matroid}, and we denote by $M(A)$ its underlying matroid. Matroids arising in this way are called \defemph{realizable} (over $K$). As in \cite[Example 4.1.2]{brandt2021tropicalflag}, if $L_1 \subseteq L_2$ are two linear subspaces of $K^n$ generated as the row span of two matrices $A_1$ and $A_2$ respectively, the induced matroids form a quotient, $\mu(A_1) \twoheadleftarrow \mu(A_2)$. Matroid quotients arising in this way are called \defemph{realizable} (over $K$). Note that a matroid quotient $M \twoheadleftarrow N$ of two realizable matroids is not necessarily realizable (see \cite[§1.7.5, Example 7]{borovik2003coxetermatroids}).
	
	\begin{definition}\label{def: valuated circuit}
		Let $\mu$ be a valuated matroid of rank $r$ on $[n]$. For each $I \in \binom{[n]}{r+1}$ define an element $C_\mu(I) \in \mathbb{T}^n$ by
		\[ C_\mu(I)_i = \begin{cases}
			\mu(I \setminus i)  & i \in I, \\
			\infty & i \notin I.
		\end{cases}
		\]
		The set of \defemph{valuated circuits} $\mathcal{C}(\mu)$ of $\mu$ is defined as the image in $\mathbb{P}(\mathbb{T}^n)$ of the following set:
		\[ \left\{ C_\mu(I) : I \in \binom{[n]}{r+1} \right\} \setminus \{ (\infty,\dots,\infty) \}. \] 
		A \defemph{cycle} of a matroid is a union of circuits. A \defemph{vector} (or \defemph{valuated cycle}) of $\mu$ is any element of $\mathbb{P}(\mathbb{T}^n)$ (tropically) generated by the valuated circuits. More explicitly, the family of vectors is
		\[ \mathcal{V}(\mu) = \left\{ \bigoplus_{C \in \mathcal{C}(\mu)} \lambda_C \odot C : \lambda_C \in \mathbb{T}, \lambda_C \neq \infty \right\}. \]
	\end{definition}
	
	For more information about vectors of a valuated matroid see \cite[Section 2.1]{maclagan2018tropicalideals}. The above terminology is partially motivated by the following fact: if $C \in \mathbb{P}(\mathbb{T}^n)$ is a valuated circuit of $\mu$, then the set of its nonzero coordinates is a circuit of the underlying matroid of $\mu$. The analogue is also true for cycles and vectors.
	
	\begin{definition}\label{def: tropicallinearspace}
		Let $\mu$ be a valuated matroid on $[n]$. The \defemph{tropical linear space} of $\mu$ is the tropical variety
		\[ \tropbar(\mu) = \bigcap_{C \in \mathcal{C}(\mu)} V \left( \bigoplus_{i \in [n]} C_i \odot x_i  \right) \subseteq \mathbb{P}(\mathbb{T}^n). \]
	\end{definition}
	
	\subsection{Pl\"{u}cker relations and the Flag Dressian}\label{sec: prelims flag dressian}
	
	In the following, let $K$ be a field with valuation $\val:K \rightarrow \mathbb{T}$.
	
	\begin{definition}
		Let $r \leq s \leq n$ be nonnegative integers. The \defemph{incidence Pl\"{u}cker relations} are the polynomials in the variables $\{p_I : I \in \binom{[n]}{r} \} \cup \{ p_J : J \in \binom{[n]}{s} \}$ with coefficients in $K$:
		\[ \mathscr{P}_{r,s;n} = \left\{ \sum_{j \in J \setminus I} \sign(j;I,J) p_{I \cup j} p_{J \setminus j} : I \in \binom{[n]}{r-1}, J \in \binom{[n]}{s+1} \right\},  \]
		where $\sign(j;I,J) = (-1)^{ \#\{ j' \in J : j < j' \} + \# \{ i \in I : i > j \}  }$. The tropicalizations of the incidence Pl\"{u}cker relations are denoted by $\mathscr{P}^{\trop}_{r,s;n}$. If $r=s$, we recover the \emph{Grassmann-Pl\"{u}cker relations} $\mathscr{P}_{r;n} := \mathscr{P}_{r,r;n}$, defining the Pl\"{u}cker embedding of the Grassmannian $G(r;n)$. We denote their tropicalizations by $\mathscr{P}^{\trop}_{r;n}$.
		
		The incidence-Pl\"ucker relations, combined with the Grassmann-Pl\"{u}cker relations, are the equations of flag varieties, defined as follows. Let $r_1 \leq \dots \leq r_k \leq n$ be nonnegative integers, and set $\mathbf{r} = (r_1,\dots,r_k)$. The \defemph{flag variety} is the subvariety of $\mathbb{P}^{\binom{n}{r_1}-1} \times \dots \times \mathbb{P}^{\binom{n}{r_k}-1}$ given as
		\[ \Fl(\mathbf{r};n) = V \big( \{ \mathscr{P}_{r_i;n} \}_{1 \leq i \leq k} \cup \{ \mathscr{P}_{r_i,r_j;n} \}_{1 \leq i < j \leq k} \big). \]
	\end{definition}
	
	\begin{remark}\label{rmk:contrasting dressian and grassmannian}
		Grassmannians and flag varieties have two respective tropical analogues. \emph{Tropical Grassmannians} $\trop(G(r;n))$ and \emph{tropical flag varieties $\trop(\Fl(\mathbf{r};n))$} are tropicalizations of their classical analogues, as in Section \ref{sec: prelims tropgem}. These \emph{tropical varieties} parametrize \emph{tropicalized} objects: $\trop(G(r;n))$ parametrizes tropicalizations of linear subspaces of $K^n$ of dimension $r$ (see \cite[Theorem 3.8]{speyersturmfels2004grassmannian}) and $\trop(\Fl(\mathbf{r};n))$ parametrizes realizable tropical flags $\tropbar(L_1) \subseteq \dots \subseteq \tropbar(L_k)$ where $L_1 \subseteq \dots \subseteq L_k$ is a flag of subspaces of $K^n$ satisfying $\dim L_i = r_i$, \cite[Remark 4.2.6]{brandt2021tropicalflag}.
		
		\emph{Dressians} $\Dr(r;n)$ and \emph{flag Dressians} $\FlDr(\mathbf{r};n)$ are the intersections of the tropical hypersurfaces given by their respective Pl\"ucker relations. They are \emph{tropical prevarieties} and parametrize \emph{tropical} objects. In general,  (flag) Dressians and tropical Grassmannians are different polyhedral complexes, see, for instance \cite{herrmann2014dressians} for the Grassmannian and \cite[Section 5]{brandt2021tropicalflag} for the flag variety. The Dressian $\Dr(r;n)$ parametrizes (not necessarily realizable) tropical linear spaces of rank $r$ in $\mathbb{P}(\mathbb{T}^n)$ as in Definition \ref{def: tropicallinearspace}, see also \cite{herrmann2014dressians}. Flag Dressians parametrize (not necessarily realizable) flags of tropical linear spaces:
	\end{remark}

	\begin{theorem}[{\cite[Theorem 1]{haque2012tropical},\cite[Lemma 2.6]{mundinger2018image},\cite[Theorem A]{brandt2021tropicalflag}}]\label{thm: brandt equivalence for tropical flags}
		Let $\boldsymbol{\mu} = (\mu_1,\dots,\mu_k)$ be a sequence of valuated matroids on a common ground set $[n]$ of ranks $\boldsymbol{r}=(r_1,\dots,r_k)$ respectively. The following are equivalent:
		\begin{enumerate}[noitemsep, label = (\alph*)]
			\item $\boldsymbol{\mu}$ is a point in $\FlDr(\boldsymbol{r};n) = \bigcap_{1 \leq i \leq k} V(\mathscr{P}^{\trop}_{r_i;n}) \cap \bigcap_{1 \leq i < j \leq k}V(\mathscr{P}^{\trop}_{r_i,r_j;n})$,
			\item $\boldsymbol{\mu}$ is a valuated flag matroid,
			\item $\tropbar(\mu_1) \subseteq \dots \subseteq \tropbar(\mu_k)$.
		\end{enumerate}
	\end{theorem}
	\begin{figure}[ht]
		\centering\begin{tabular}{c c}
			
			\begin{tikzpicture}
				
				\draw[densely dotted] (0,0)--(-2.6,-2)
				(0,0) -- (-0.2, 2.5)
				(0,0) -- (0.2, 0.5)
				(0,0) -- (2.6,-1);
				
				\node[below left] at (-2.6,-2) {\small$1$};
				\node[above left] at (-0.2, 2.5) {\small$3$};
				\node[ above right] at (0.2, 0.5){\small$4$};
				\node[ below right] at (2.6,-1){\small$2$};
				
				\draw[cyan, thick] (-1.3,-2.5) -- (0,-1.5) -- (1.3,-2)  (0,-1.5) -- (0,0) -- (0, 1.5) -- (-0.1, 2.75) (0,1.5) -- (0.1,1.75);
				\fill[cyan] (0, -1.5) circle (1.5pt);
				\fill[cyan]  (0, 1.5)  circle (1.5pt);
				\draw[red, thick] (-2.7,-0.75) -- (-1.4, 0.25) -- (-1.5,1.5)  (-1.4, 0.25)-- (0,0) --(1.4, -0.25) -- (2.7,-0.75) (1.4, -0.25) -- (1.5,0);
				\fill[red] (-1.4, 0.25) circle (1.5pt);
				\fill[red] (1.4, -0.25) circle (1.5pt);
				\fill (-0.65,-2) circle (2.5pt);
				\fill (-0.7,0.125) circle (2.5pt);
			\end{tikzpicture}&
			\begin{tikzpicture}
				\foreach \i in {1,...,5}
				\fill (\i*360/5:1.5) coordinate (5\i) circle(2 pt);
				
				\foreach \i in {1,...,5}
				\fill (\i*360/5:3) coordinate (5\i) circle(2 pt);
				\draw  (360/5:3) --  (2*360/5:3) --  (3*360/5:3) --  (4*360/5:3) --  (5*360/5:3)-- (360/5:3);
				\node[above] at (360/5:1.5) {$1$};
				\node[above] at (2*360/5:1.5) {$4$};
				\node[above, left] at (3*360/5:1.5) {$13$};
				\node[above, right] at (4*360/5:1.5) {$14$};
				\node[above] at (5*360/5:1.5) {$24$};
				\node[above] at (360/5:3) {$12$};
				\node[above, left] at (2*360/5:3) {$34$};
				\node[above, left] at (3*360/5:3) {$3$};
				\node[below] at (4*360/5:3) {$23$};
				\node[above,right] at (5*360/5:3) {$2$};
				\draw  (360/5:3) -- (360/5:1.5);
				\draw  (2*360/5:3) -- (2*360/5:1.5);
				\draw  (3*360/5:3) -- (3*360/5:1.5);
				\draw[blue]  (4*360/5:3) -- (4*360/5:1.5);
				\draw  (5*360/5:3) -- (5*360/5:1.5);
				\draw  (3*360/5:1.5) -- (5*360/5:1.5) -- (2*360/5:1.5)-- (4*360/5:1.5)-- (1*360/5:1.5)-- (3*360/5:1.5);
				\draw[blue] (3*360/5:1.5) -- (5*360/5:1.5);
				\draw[blue]  (1*360/5:3) -- (2*360/5:3);
				\foreach \i in {1,...,5}
				\fill (\i*360/5:1.5) coordinate (5\i) circle(2 pt);
				
				\foreach \i in {1,...,5}
				\fill (\i*360/5:3) coordinate (5\i) circle(2 pt);
				\fill[red] (-0.55,-0.65) circle (2.5pt);
				\fill[cyan] (0.7,2.125) circle (2.5pt);
			\end{tikzpicture}	
		\end{tabular}
		\caption{On the left, two tropical flags in $\trop ( \Fl(1,2;4) )$ in $\mathbb{P}(\mathbb{T}^4)$. On the right, a Petersen graph representing the tropical flag variety $\trop ( \Fl(1,2;4) )$ after quotienting by its lineality space. The red point corresponds to the red flag on the left, the cyan point is the cyan flag.}
		\label{fig: tropfl24}
	\end{figure}
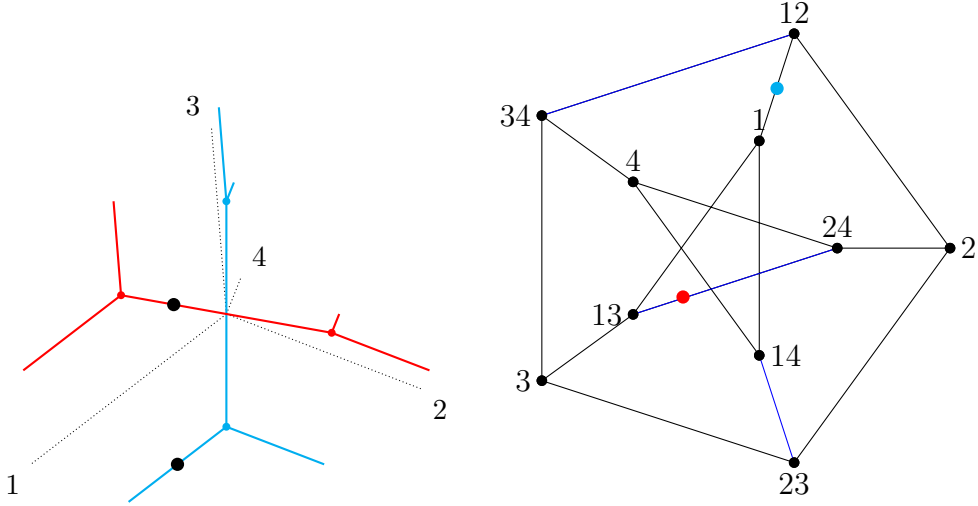
	
	\begin{example}\label{ex: Fl 1,2;4}
		In this example we describe the tropicalization of the flag variety 
		$\Fl(1,2;4)$, parametrizing flags of points in lines in $\mathbb{P}^3$, with respect to the trivial valuation. By definition, $\Fl(1,2;4) = V ( \mathscr{P}_{2;4} \cup \mathscr{P}_{1,2;4} ),$
		since $\mathscr{P}_{1;4}$ contains just the zero polynomial. 
		The tropicalizations of the equations defining the ideal, given below, form a tropical basis, thus $\FlDr(1,2;4)=\trop ( \Fl(1,2;4) )$ (see \cite[Theorem 5.2.1]{brandt2021tropicalflag}).
		\begin{align*}
			\mathscr{P}_{2;4}^{\trop} & = \left\{ p_{1,4}p_{2,3} \oplus p_{1,3}p_{2,4} \oplus p_{1,2}p_{3,4} \right\}, \\
			\mathscr{P}_{1,2;4}^{\trop} & = \left\{ \begin{array}{l}
				p_{1}p_{2,3} \oplus p_{2}p_{1,3} \oplus p_3p_{1,2}, \\
				p_{4}p_{1,2} \oplus p_2p_{1,4} \oplus p_1p_{2,4}, \\
				p_{4}p_{1,3} \oplus p_1p_{3,4} \oplus p_3p_{1,4}, \\
				p_{4}p_{2,3} \oplus p_2p_{3,4} \oplus p_3p_{2,4}.
			\end{array} \right\}
		\end{align*}
		The tropicalization $\trop \big( \Fl(1,2;4) \big)$ can be computed in Macaulay2 \cite{M2} and is a 5-dimensional simplicial fan in ${\mathbb{P}(\mathbb{T}^4)}\times \mathbb{P}(\mathbb{T}^6)$ with a lineality space of dimension 3 and an f-vector $(1,10,15)$ after taking the quotient by the lineality space. The tropical variety modulo lineality space is a Petersen graph (see Figure \ref{fig: tropfl24}). 
		
		A point $p$ in a top-dimensional cone (corresponding to an edge of the Petersen graph) parametrizes a point $v_p$ contained in a generic tropical line $L_p$ in the three-dimensional tropical projective space ${\mathbb{P}(\mathbb{T}^4)}$. A generic tropical line in ${\mathbb{P}(\mathbb{T}^4)}$ is a one dimensional, balanced polyhedral complex with one bounded edge and four unbounded edges, two adjacent to each vertex of the bounded edge. The four distinct directions of the unbounded edges of $L_p$ are given by the images of $e_i$ in ${\mathbb{P}(\mathbb{T}^4)}$ for $i \in [4]$. Since $L_p$ is balanced, the directions of the unbounded edges are completely determined by the direction of the bounded edge, which is of the form $e_i+e_j$ for $i,j \in [4]$ distinct. As $e_a+e_b = -(e_c+e_d)$ in ${\mathbb{P}(\mathbb{T}^4)}$ for all $a,b,c,d \in [4]$ distinct, there are three such choices. The position of the point $v_p$ on the line $L_p$ can be freely chosen, and there are 5 choices, one for each edge of $L_p$. This amounts to $3 \times 5 = 15$ possible choices, corresponding to the maximal cones.
		
		In the Petersen graph found in Figure \ref{fig: tropfl24}, these cones are indexed as follows. The three (blue) edges connecting the vertices $(ab)$ to $(cd)$, for $a,b,c,d \in [4]$ distinct, correspond to the three cases where the point lies on the bounded edge in direction $e_a+e_b$, for an example see the red flag in Figure \ref{fig: tropfl24}. The remaining twelve (black) edges, connecting $(a)$ to $(ab)$ for $a,b\in[4]$ distinct, correspond to the cases where the point lies on an unbounded edge in direction $e_a$, and the bounded edge is in direction $e_a+e_b$. For an example, see the cyan flag in Figure \ref{fig: tropfl24}.
		
		The length of the bounded edge of $L_p$ and the position of the point $v_p$ in $L_p$ can be freely chosen, amounting to two degrees of freedom. These are the two degrees of freedom remaining after taking the quotient with the lineality space. The choice of the point $v_p$ in  ${\mathbb{P}(\mathbb{T}^4)}$ corresponds to three further degrees of freedom. In total, this generates a cone of (projective) dimension five. 
	\end{example}
	
	\section{Linear Degenerate Flag Dressian}\label{sec: linear degenerate section}
	
	In this section, we define the linear degenerate flag Dressian and prove the equivalences of our main results, Theorem \ref{thm:main_valuated_ld_flag_correspondence} and Theorem \ref{thm:main_valuated_ld_flag_correspondence_realizable}.

	\subsection{Linear degenerate Pl\"{u}cker relations}\label{sec: ld_pluecker_relations}
	
	We start by defining the linear degenerate Pl\"{u}cker relations. These parametrize linear degenerate flags, which are collections of subspaces $V_1, \dots, V_m\subseteq K^n$ satisfying $\pr_{S_i}(V_i)\subseteq V_{i+1}$ for sets $S_i\subseteq [n]$ and all $i\in \{1,\dots, m\}$, where for a subset of indices $S \subseteq [n]$, we define the linear map $\pr_S:K^n \rightarrow K^n$ by setting $\pr_S(e_i) = 0$ if $i \in S$ and $\pr_S(e_i)=e_i$ otherwise.
	
	\begin{definition}[Linear degenerate Pl\"{u}cker relations]\label{def: ld pluecker relations}
		Let $r \leq s  \leq n$ be nonnegative integers and let $S \subseteq [n]$. The \defemph{linear degenerate Pl\"{u}cker relations} are the following polynomials in the variables $\{ p_I : I \in \binom{[n]}{r} \} \cup \{ p_J : J \in \binom{[n]}{s} \}$ with coefficients in $K$: 
		\[ \mathscr{P}_{r,s;S;n} = \left\{ \sum_{j \in J \setminus (I \cup S)} \sign(j;I,J) p_{I \cup j} p_{J \setminus j} : I \in \binom{[n]}{r-1}, J \in \binom{[n]}{s+1}  \right\} \]
		where $\sign(j;I,J) = (-1)^{ \#\{ j' \in J : j < j' \} + \# \{ i \in I : i > j \}  }$. We denote their tropicalizations by $\mathscr{P}_{r,s;S;n}^{\trop}$.
	\end{definition}
	
	Linear degenerate Pl\"ucker relations appear in \cite[Section 5.1]{fourier2020lineardegenerations}, arising as initial degenerations of the Pl\"ucker relations. The form in which we are expressing them here can be deduced from the relations given in \cite{lorscheid2019pluckerelations}. The linear degenerate Pl\"{u}cker relations parametrize linear degenerate flags of linear spaces. For the sake of completeness, we give a proof similar to the original proof of the classical Pl\"ucker relations following \cite[Theorem 1.8]{bokowski1989pluecker}.
	
	\begin{proposition}\label{prop: linear degenerate pluecker relations}
		Let $U$ and $V$ be vector subspaces of $K^n$ of dimension $r \leq s$ respectively, and let $S \subseteq [n]$. We have $\pr_S(U) \subseteq V$ if and only if the Pl\"{u}cker coordinates of $U$ and $V$ satisfy the linear degenerate Pl\"{u}cker relations $\mathscr{P}_{r,s;S;n}$.
	\end{proposition}
	\begin{proof}
		Suppose that $\pr_S(U) \subseteq V$. Let $A \in K^{r,n}$ be a matrix whose rows are a basis of $U$, and let $A' \in K^{r,n}$ be the matrix obtained from $A$ by substituting the columns indexed by $S$ with columns of zeros. Note that the rows of $A'$ are a set of generators for $\pr_S(U)$. Let $B \in K^{s,n}$ be a matrix whose rows are a basis of $V$, obtained by extending a basis of $\pr_S(U)$ consisting of rows of $A'$. Fix $I = \{ i_1 < \dots < i_{r-1} \} \in \binom{[n]}{r-1}$ and $J = \{ j_1 < \dots < j_{s+1} \}  \in \binom{[n]}{s+1}$. The column vectors $B_{j_1},\dots,B_{j_{s+1}}$ are linearly dependent, satisfying the dependency relation
		\[ \sum_{k=1}^{s+1} (-1)^k \det(B_{j_1},\dots,B_{j_{k-1}},B_{j_{k+1}},\dots,B_{j_{s+1}}) \cdot B_{j_k} = 0. \]
		In particular, from the construction of $A'$ and $B$ we also obtain
		\[ \sum_{k=1}^{s+1} (-1)^k \det(B_{j_1},\dots,B_{j_{k-1}},B_{j_{k+1}},\dots,B_{j_{s+1}}) \cdot A'_{j_k} = 0. \]
		Substituting the above expression of the ($r$-dimensional) zero vector into the equation $\det(0,A_{i_1},\dots,A_{i_{r-1}}) = 0$ we obtain
		\begin{align*}
			& \det \bigg( \sum_{k = 1}^{s+1} (-1)^k \det(B_{j_1}, \dots, B_{j_{k-1}}, B_{j_{k+1}}, \dots, B_{j_{s+1}}) \cdot A'_{j_k}, A_{i_1},\dots,A_{i_{r-1}} \bigg) = \\
			& \sum_{k = 1}^{s+1} (-1)^k \det(B_{j_1}, \dots, B_{j_{k-1}}, B_{j_{k+1}}, \dots, B_{j_{s+1}}  ) \cdot \det(A'_{j_k},  A_{i_1},\dots,A_{i_{r-1}}) = 0.
		\end{align*}
		Now by construction we have
		\[ A'_{j_k} = \begin{cases}
			A_{j_k} & \text{if } j_k \notin S, \\
			0 & \text{if } j_k \in S.
		\end{cases} \]
		Thus, by substituting the above in the previously displayed equation, and by reordering the column vectors keeping track of sign changes, the above relations become the desired linear degenerate Pl\"{u}cker relations, up to a possible (global) change of sign that depends on $r$ and $s$.
		
		Conversely, suppose that the Pl\"{u}cker coordinates of $U$ and $V$ satisfy the linear degenerate Pl\"{u}cker relations. Let $A$ and $A'$ be as above, and let $B \in K^{s,n}$ be a matrix whose rows are a basis of $V$. We need to show that the rows of $A'$ are spanned by the rows of $B$. Let $I = \{ i_1 < \dots < i_{r-1} \} \in \binom{[n]}{r-1}$ and $J = \{ j_1 < \dots < j_{s+1} \}  \in \binom{[n]}{s+1}$. Proceeding similarly as above, from the incidence Pl\"{u}cker relations we can write
		\begin{equation}\label{eq: determinant plucker relations}
			\det \bigg( \sum_{k = 1}^{s+1} (-1)^k \det(B_{j_1}, \dots, B_{j_{k-1}}, B_{j_{k+1}}, \dots, B_{j_{s+1}}) \cdot A'_{j_k}, A_{i_1},\dots,A_{i_{r-1}} \bigg) = 0.
		\end{equation}
		Since $A$ has maximal rank, we can choose a subset $I' \in \binom{[n]}{r}$ such that the columns of $A$ indexed by $I'$ form a basis. By choosing all possible cardinality $r-1$ subsets $I \subseteq I'$ in \eqref{eq: determinant plucker relations}, we have that the first vector in the argument of the determinant in \eqref{eq: determinant plucker relations} is in the span of the spaces generated by the vectors indexed by all such sets $I$. This is possible only for the zero vector. Therefore, we obtain
		\[ \sum_{k = 1}^{s+1} (-1)^k \det(B_{j_1}, \dots, B_{j_{k-1}}, B_{j_{k+1}}, \dots, B_{j_{s+1}}) \cdot A'_{j_k} = 0. \]
		Let $C$ be the matrix consisting of the rows of $B$ plus an additional row of $A'$. By using Laplace expansion for computing the determinant of the square submatrix of $C$ with columns indexed by $J$ with respect to the row of $A'$, the above dependencies imply that the rank of $C$ is equal to the rank of $B$, i.e., that the row of $A'$ in $C$ is a linear combination of the rows of $B$.
	\end{proof}
	
	Let $r_1 \leq \dots \leq r_k \leq n$ be nonnegative integers, let $S_1,\dots,S_{k-1} \subseteq [n]$, and write $\mathbf{r} = (r_1,\dots,r_k)$, $\mathbf{S} = (S_1,\dots,S_{k-1})$, and $S_{ij} = S_i \cup S_{i+1} \cup \dots \cup S_{j-1}$ for $1 \leq i < j \leq k$.
	
	\begin{definition}[Linear degenerate flag variety]
		The \defemph{linear degenerate flag variety} of rank $\mathbf{r}$ and degeneration type $\mathbf{S}$ is the following subvariety of $\mathbb{P}^{\binom{n}{r_1}-1} \times \dots \times \mathbb{P}^{\binom{n}{r_k}-1}$
		\[  \LFl(\mathbf{r},\mathbf{S};n) = V\big(\{\mathscr{P}_{r_i;n} \}_{1 \leq i \leq k} \cup \{ \mathscr{P}_{r_i,r_j;S_{ij};n}\}_{1 \leq i < j \leq k} \big). \]
		We call its tropicalization the \defemph{linear degenerate tropical flag variety}. 
	\end{definition}
	
	A consequence of Proposition \ref{prop: linear degenerate pluecker relations} is $(a) \Leftrightarrow (b)$ of Theorem \ref{thm:main_valuated_ld_flag_correspondence_realizable}. Before proving it, we introduce some more notation to give a matroidal interpretation of Proposition \ref{prop: linear degenerate pluecker relations}. First, we recall the definition of \emph{deletion} for valuated matroids.
	
	\begin{proposition-definition}[{\cite[Proposition 1.2]{dress1992valuated}}]\label{propdef: deletion of valuated matroid}
		Let $\mu$ be a matroid of rank $r$ on $[n]$ with underlying matroid $M$, and let $S \subseteq [n]$. Let $k$ be the rank of the deletion $M \setminus S$. Choose $I \in \binom{S}{r-k}$ such that $([n] \setminus S) \cup I$ has rank $r$. Then, the map $\mu \setminus S : \binom{[n] \setminus S}{k} \rightarrow \mathbb{T}$ defined by $(\mu \setminus S)(B) = \mu(B \cup I)$ is a valuated matroid, with underlying matroid $M \setminus S$. Further, $\mu \setminus S$ is compatible with equivalence, and different choices of $I$ give rise to equivalent valuated matroids. The matroid $\mu \setminus S$ is called the \emph{deletion} of $\mu$ by $S \subseteq [n]$.
	\end{proposition-definition}

	We write $\mu_S$ for the valuated matroid on the ground set $[n]$ defined by  
	\begin{equation}\label{eq: def of mu_S}
		\mu_S(B) = \begin{cases}
			(\mu \setminus S)(B) & \text{if } B \cap S = \emptyset, \\
			\infty & \text{otherwise,}
		\end{cases}
	\end{equation}
	for every $B \in \binom{[n]}{\rk(\mu\setminus S)}$. The matroid $\mu_S$ can be alternatively regarded as a direct sum of valuated matroids (defined as in \cite[Definition 2.6]{husic2022complete}):
	\[ \mu_S = (\mu \setminus S) \oplus U_{0,|S|}, \]
	where we add the deleted elements of the ground set as loops.
	
	Recall that we can view (an equivalence class of) a valuated matroid $\mu: \binom{[n]}{r} \rightarrow \mathbb{T}$ as a point $\mu \in \mathbb{P}\left(\mathbb{T}^{\binom{n}{r}}\right)$, and similarly a pair of valuated matroids $(\mu,\nu)$ of rank $r$ and $s$ respectively, as a point $\mu \times \nu \in \mathbb{P}\left(\mathbb{T}^{\binom{n}{r}}\right) \times \mathbb{P}\left(\mathbb{T}^{\binom{n}{s}}\right)$. We are now ready to prove $(a) \Leftrightarrow (b)$ of Theorem \ref{thm:main_valuated_ld_flag_correspondence_realizable}.

	\begin{corollary}\label{cor: thm b proof}
		Let $\mu$ and $\nu$ be two realizable valuated matroids of rank $r$ and $s$ on the common ground set $[n]$, and let $S \subseteq [n]$. Then, $\mu \times \nu \in \tropbar(\LFl(r,s,S;n))$ if and only if $\mu_S \twoheadleftarrow \nu$ is a realizable quotient of valuated matroids.
	\end{corollary}
	\begin{proof}
		For a definition of realizable quotient of valuated matroids, see the paragraph after Definition \ref{def: valuated flag matroid}.
		
		If $\mu \times \nu \in \tropbar(\LFl(r,s,S;n))$, from the Fundamental Theorem of Tropical Geometry \cite{maclagan2015book}, there exist realizations $U$ of $\mu$ and $V$ of $\nu$ such that the Pl\"{u}cker coordinates of $U$ and $V$ are a point of $\LFl(r,s,S;n)$. From Proposition \ref{prop: linear degenerate pluecker relations} this implies that $\pr_S(U) \subseteq V$. Now note that, by definition, the valuated matroid of $\pr_S(U)$ is $\mu_S$, therefore, from Theorem \ref{thm: brandt equivalence for tropical flags} we have that $\pr_S(U) \subseteq V$ implies $\mu_S \twoheadleftarrow \nu$ and the last quotient is realizable.
		
		Conversely, assume that $\mu_S \twoheadleftarrow \nu$ is a realizable quotient. By definition, $\mu_S$ is realizable. Since $\mu$ and $\nu$ are both realizable by assumption, this means that there exist realizations $U$ of $\mu$ and $V$ of $\nu$ are such that $\pr_S(U) \subseteq V$.
		From Proposition \ref{prop: linear degenerate pluecker relations} this implies that the Pl\"{u}cker coordinates of $U$ and $V$ satisfy the linear degenerate Pl\"{u}cker relations. Therefore the Pl\"ucker coordinates of their valuated matroids $\mu$ and $\nu$ are tropicalization of the Pl\"{u}cker coordinates, that is $\mu \times \nu \in \tropbar(\LFl(r,s,S;n))$.
	\end{proof}

	\begin{definition}[Linear degenerate flag Dressian]\label{def:ld_flag_dressian}
		The \defemph{linear degenerate flag Dressian} of rank $\mathbf{r}$ and degeneration type $\mathbf{S}$ is the tropical prevariety
		\[  \LFlDr(\mathbf{r},\mathbf{S};n) \subseteq \mathbb{P}(\mathbb{T}^{\binom{n}{r_1}}) \times \dots \times \mathbb{P}(\mathbb{T}^{\binom{n}{r_k}}) \]
		given by the intersection of the tropical hypersurfaces of the tropical polynomials in $\{ \mathscr{P}^{\trop}_{r_i;n} \}_{1 \leq i \leq k} \cup \{ \mathscr{P}^{\trop}_{r_i,r_j;S_{ij};n} \}_{1 \leq i < j \leq k}$.
	\end{definition}
	
	Now, for $S \subseteq [n]$, define the projection map $\pr_S^{\trop}: \mathbb{T}^n \rightarrow \mathbb{T}^n$ by
	\[ \bigg( \pr_S^{\trop}(x_1,\dots,x_n) \bigg)_i = \begin{cases}
		x_i & \text{if } i \notin S, \\
		\infty & \text{if } i \in S.
	\end{cases} \]
	The projection $\pr_S^{\trop}$ does not give us a well-defined map on the tropical projective space $\mathbb{P}(\mathbb{T}^n)$. Denote by $\varphi:\mathbb{T}^n \setminus \{ (\infty, \dots, \infty) \} \rightarrow \mathbb{P}(\mathbb{T}^n)$ the natural quotient map. By abuse of notation, for a subset $X \subseteq \mathbb{P}(\mathbb{T}^n)$ we set
	\[ \pr_S^{\trop}(X) = \varphi \Big( \pr_S^{\trop} \big( \varphi^{-1}(X) \big) \setminus \{ (\infty,\dots,\infty) \} \Big). \]
	Alternatively, $\pr_S^{\trop}(X)$ can be equivalently defined as the image of $X \subseteq \mathbb{P}(\mathbb{T}^n)$ under the restriction of $\pr_S^{\trop}$ where it is defined on $\mathbb{P}(\mathbb{T}^n)$.
	
	\begin{example}\label{ex: linear degenerate section 3}
		In this example we describe the tropicalization of the linear degenerate flag variety 
		$\LFl((1,2),\{ 1 \};4)$, parametrizing tuples $(v,L)$ of points $v$ and lines $L$ in $\mathbb{P}^3$ such that $\pr_1(v)\subseteq L$. By definition
		$ \LFl((1,2),\{1\};4) = V ( \mathscr{P}_{2;4} \cup \mathscr{P}_{1,2;\{1\};4} ), $
		since $\mathscr{P}_{1;4}$ contains just the zero polynomial. We verify computationally that the tropicalizations of the linear degenerate Pl\"ucker relations given below form a tropical basis, i.e. generate $\trop(\LFl((1,2),\{ 1 \};4))$.
		\begin{align*}
			\mathscr{P}_{2;4}^{\trop} & = \left\{ p_{1,4}p_{2,3} \oplus p_{1,3}p_{2,4} \oplus p_{1,2}p_{3,4} \right\}, \\
			\mathscr{P}_{1,2;\{1\};4}^{\trop} & = \left\{ \begin{array}{l}
				p_{3}p_{1,2} \oplus p_{2}p_{1,3}, \\
				p_{4}p_{1,2} \oplus p_2p_{1,4}, \\
				p_{4}p_{1,3} \oplus p_3p_{1,4}, \\
				p_{4}p_{2,3}\oplus p_3p_{2,4} \oplus p_2p_{3,4}.
			\end{array} \right\}
		\end{align*}
		Note that the polynomials in $\mathscr{P}_{1,2;\{1\};4}^{\trop}$ are obtained from those in $\mathscr{P}_{1,2;4}^{\trop}$ by deleting all monomials containing $p_1$ (compare with Example \ref{ex: Fl 1,2;4}).
		The tropicalization $\trop \big( \LFl((1,2),\{1\};4) \big)$ can be computed in Macaulay2 \cite{M2} and is a 5-dimensional simplicial fan in ${\mathbb{P}(\mathbb{T}^4)}\times {\mathbb{P}(\mathbb{T}^6)}$. Its lineality space has dimension $4$ and the quotient of the variety by the lineality space has f-vector $(1,3)$.
		A point $p$ in the tropical linear degenerate flag variety $\trop \big( \LFl((1,2),\{1\};4) \big)$ corresponds to an arrangement of a point $v_p = (v_1:v_2:v_3:v_4)$ and a general tropical line $L_p$ in ${\mathbb{P}(\mathbb{T}^4)}$ such that $\pr_{e_1}(v_p)\in L_p$. This last condition implies that $L_p$ contains the point $(\infty:v_2:v_3:v_4)$. This occurs only if $L_p$ has an unbounded edge $l_1$ in (projective) coordinate direction $e_1$ at $(x:v_2:v_3:v_4)$ for some $x\in\mathbb{R}$. The direction of the unbounded edge adjacent to $l_1$ can be chosen to be any of (projective) coordinate directions $e_2, e_3$ and $e_4$. The balancing condition fixes the directions of the remaining unbounded edges, as in Example \ref{ex: Fl 1,2;4}. The three choices for the direction vector correspond to the three maximal cones.
	\end{example}
	
	\begin{definition}[Linear degenerate tropical flag]
		A \defemph{linear degenerate tropical flag} of degeneration type $\mathbf{S} = (S_1,\dots,S_{k-1})$, with $S_i \subseteq [n]$, is a sequence of tropical linear spaces $(\tropbar(\mu_1),\dots,\tropbar(\mu_k))$ in $\mathbb{P}(\mathbb{T}^n)$ such that for all $i\in \{1,\dots, k-1\}$ we have $\pr_{S_i}^{\trop}(\tropbar(\mu_i)) \subseteq \tropbar(\mu_{i+1})$.
	\end{definition}
	
	\noindent
	A picture of a linear degenerate tropical flag can be found in Figure \ref{fig: tropical flag}(b).
	
	To show that points in the linear degenerate flag Dressian parametrize linear degenerate tropical flags, and thus $(a) \Leftrightarrow (c)$ in Theorem \ref{thm:main_valuated_ld_flag_correspondence}, we give an equivalent definition of tropical linear spaces in terms of \emph{cocircuits}.
	
	\begin{definition}\label{def: valuated cocircuit}
		The \defemph{dual} of a valuated matroid $\mu$ is the valuated matroid $\mu^*$ defined by $\mu^*(I) = \mu([n] \setminus I)$ for all $I \in \binom{[n]}{d}$. The \defemph{valuated cocircuits} of $\mu$ are the valuated circuits of $\mu^*$. For each $I \in \binom{[n]}{r-1}$ define $C^*_\mu(I) \in \mathbb{T}^n$ by
		\[ C^*_\mu(I)_i = \begin{cases}
			\mu(I \cup i)  & i \notin I, \\
			\infty & i \in I.
		\end{cases}\]
		The set of \defemph{valuated cocircuits} $\mathcal{C}^*(\mu)$ is the image in $\mathbb{P}(\mathbb{T}^n)$ of the following set:
		\[ \left\{ C^*_\mu(I) : I \in \binom{[n]}{r-1} \right\} \setminus \{ (\infty,\dots,\infty) \}. \] 
	\end{definition}
	
	Let $\mu$ be a valuated matroid on $[n]$. Its \emph{tropical linear space} can be equivalently defined as the span of its cocircuits (see, for instance,  \cite[Theorem B]{brandt2021tropicalflag}):
	\begin{equation}\label{eq: trop from cocircuits}
		\tropbar(\mu) = \displaystyle \left\{ \bigoplus_{C \in \mathcal{C}^*(\mu)} \lambda_C \odot C : \lambda_C \in \mathbb{T}, \lambda_C \neq \infty \right\}.
	\end{equation}

	\begin{proposition}\label{prop: b iff c}
		Let $\mu$ and $\nu$ be two valuated matroids on a common ground set $[n]$, of rank $r$ and $s$ respectively, and let $S \subseteq [n]$. The following statements are equivalent:
		\begin{enumerate}[label=(\arabic*)]
			\item $\mu \times \nu \in \LFlDr(r,s,S;n)$,
			\item $\pr_S^{\trop}(\tropbar(\mu)) \subseteq \tropbar(\nu)$.
		\end{enumerate}
	\end{proposition}
	\begin{proof}
		Both $\mu$ and $\nu$ satisfy their respective tropical Grassmann-Pl\"{u}cker relations if and only if $\mu$ and $\nu$ are valuated matroids.
		
		Now $\mu \times \nu \in\LFlDr(r,s,S;n)$ if and only if for every $I \in \binom{[n]}{r-1}$, $J \in \binom{[n]}{s+1}$ the minimum in
		\begin{equation*}
			\bigoplus_{j \in J \setminus (I \cup S)} p_{I \cup j} p_{J \setminus j}
		\end{equation*}
		is achieved at least twice. Since, from Definition \ref{def: valuated cocircuit}, we have
		\[ \bigg(\pr_S^{\trop} \big( C^*_\mu(I) \big) \bigg)_j = \begin{cases}
			\mu(I \cup j) & \text{if } j \notin I \cup S \\
			\infty & \text{otherwise},
		\end{cases} \]
		the above statement is equivalent to requiring that, for every  $I \in \binom{[n]}{r-1}$ and $J \in \binom{[n]}{s+1}$, the minimum in
		\[ \begin{split}
			\bigoplus_{j \in J \setminus (I \cup S)} \mu(I \cup j) \odot \nu(J \setminus j) & = \bigoplus_{j \in J \setminus (I \cup S)} C_\mu^*(I)_j \odot C_\nu(J)_j \\
			& = \bigoplus_{j \in [n]} \bigg(\pr_S^{\trop} \big( C^*_\mu(I) \big) \bigg)_j \odot C_\nu(J)_j
		\end{split} \]
		is achieved at least twice. This holds true if and only if for every valuated cocircuit $C_\mu^*(I)$ of $\mu$ and every valuated circuit $C_\nu(J)$ of $\nu$, we have
		\[ \pr_S(C^*_\mu(I)) \in V \left( \bigoplus_{j \in [n]} C_\nu(J)_j \odot x_i\right). \]
		The above statement is equivalent to $\pr_S(C^*_\mu(I)) \in \tropbar(\nu)$ for every valuated cocircuit $C^*_\mu(I)$. By \eqref{eq: trop from cocircuits} and the fact that tropical linear spaces are tropically convex, see \cite{develin2004tropical}, this is equivalent to $\pr_S^{\trop}(\tropbar(\mu)) \subseteq \tropbar(\nu)$.
	\end{proof}	    
	
	This proves  $(a) \Leftrightarrow (c)$ in Theorem \ref{thm:main_valuated_ld_flag_correspondence}.
	
	\begin{remark}\label{rmk: Diane's remark}
		One could define the linear degenerate flag Dressian with just the ``consecutive" incidence Pl\"{u}cker relations $\mathscr{P}_{r_i,r_{i+1};S_{i,i+1};n}^{\trop}$, instead of taking all the relations $\mathscr{P}_{r_i,r_{j};S_{i,j};n}^{\trop}$ for $1 \leq i < j \leq n$. One of the consequences of the previous proposition is that these two a priori different ways of defining the linear degenerate flag Dressian give rise to the same tropical prevariety.
	\end{remark}
	
	\begin{remark}\label{rmk: Grassmannians and Dressians are different}
		As observed in Remark \ref{rmk:contrasting dressian and grassmannian} for the flag Dressian and the flag variety, the tropicalization of the linear degenerate flag variety and the linear degenerate flag Dressian are, in general, different. This follows as the tropical flag variety and the flag Dressian differ for $n\geq 6$, see \cite[Example 5.2.4]{brandt2021tropicalflag}. Further, even the linear degenerate flag varieties with highest degeneration (i.e. $S_i =[n]$ for all $i$ in Definition \ref{def:ld_flag_dressian}) differ from their linear degenerate flag Dressians, as they are products of tropicalized Grassmannians and products of Dressians respectively, and Dressians and tropical Grassmannians are different for large enough $k$ and $n$. 
	\end{remark}
	\subsection{Linear degenerate valuated flag matroids} \label{sec:ld flags of projective spaces}
	
	In this section, we prove the equivalence $(b) \Leftrightarrow (c)$ of Theorem \ref{thm:main_valuated_ld_flag_correspondence} and Theorem \ref{thm:main_valuated_ld_flag_correspondence_realizable}, that is, (realizable) linear degenerate valuated flag matroids correspond to (realizable) linear degenerate tropical flags. We begin by defining linear degenerate valuated flag matroids. Again, for a valuated matroid $\mu$ on $[n]$ and a subset $S \subseteq [n]$ we are going to use our notation
	$\mu_S$ defined in \eqref{eq: def of mu_S}.	
	More explicitly, the valuated circuits of the deletion $\mu\setminus S$ are
	\begin{align}\label{thm:valuated circuits deletion contraction}
		\mathcal{C}(\mu \setminus S) & = \left\{ C_{|[n] \setminus S} : C \in \mathcal{C}(\mu), \supp(C) \subseteq [n] \setminus S \right\}. 
	\end{align} 
	This is a direct consequence of \cite[Theorem 3.1]{murota2001circuitvaluation}. For the formulation used here, see \cite[Theorem 3.1.6]{brandt2021tropicalflag}.
	
	\begin{definition}[Linear degenerate valuated flag matroid]\label{def: ld flag matroid}
		A sequence of valuated matroids $\boldsymbol{\mu} = (\mu_1,\dots,\mu_k)$  on $[n]$ is a \defemph{linear degenerate valuated flag matroid} if for all $i\in \{1,\dots, k-1\}$, there exists $S_i \subseteq [n]$ such that $(\mu_i)_{S_i} \twoheadleftarrow \mu_{i+1}$. Further, $\boldsymbol{\mu}$ is \defemph{realizable} if all $\mu_i$ are realizable and the quotients $(\mu_i)_{S_i} \twoheadleftarrow \mu_{i+1}$ are \emph{simultaneously} realizable for all $i \in \{ 1,\dots,k-1\}$, i.e. if there exists a realization $L_i$ for each $\mu_i$ such that the realization $\pr_{S_i}(L_i)$ of $(\mu_i)_{S_i}$ is contained in $L_{i+1}$, using the same realization $L_i$ for each occurence of the matroid $\mu_i$.
	\end{definition}

	\begin{proposition}
		\label{lem:projection of tropical linear space}
		Let $\mu$ be a valuated matroid on the ground set $[n]$ and let $S\subseteq [n]$. Then
		\[ \pr^{\trop}_{\{S\}}(\tropbar(\mu)) = \tropbar\left( \mu \setminus S \right) \times \{ \infty \}^{ \{ S \} } = \tropbar(\mu_S) . \]
	\end{proposition}
	\begin{proof}
		The second equality follows from the definition of $\mu_S$. In fact, let $s$ be a loop. Then, $s$ is not contained in any basis and thus, any circuit $C$ has entry $C_s = \infty$. As any vector $v$ is a (tropical) span of circuits, this implies that every $v$ has entry $v_s = \infty$. 
		
		Now we prove the first equality. First, we note that we can restrict to the case $S=\{s\}$ and obtain the result for arbitrary $S$ by inductively re-applying the one-element case.
		
		Let $v \in \tropbar(\mu)$. Then the minimum in $\{ C_i + v_i \}_{i \in [n]}$ is achieved at least twice for every $C \in \mathcal{C}(\mu)$. In particular, the minimum in $\{ C_i + v_i \}_{i \in [n] \setminus s}$ is achieved at least twice for every $C \in \mathcal{C}(\mu)$ where $\supp(C) \subseteq [n] \setminus s$. From \eqref{thm:valuated circuits deletion contraction}, $\pr_{\{s\}}^{\trop}(v) \in \tropbar(\mu \setminus s ) \times \{ \infty \}^{\{s\}}$. This proves the first inclusion.
		
		For the reverse inclusion, let $v \in \tropbar(\mu \setminus s) \times \{ \infty \}^{\{s\}}$. Then, the minimum in $\{C_i + v_i \}_{i \in [n] \setminus s}$ is achieved at least twice for every $C \in \mathcal{C}(\mu \setminus s)$. From \eqref{thm:valuated circuits deletion contraction} this means it is achieved at least twice for every $C \in \mathcal{C}(\mu)$ with $\supp(C) \subseteq [n] \setminus s$. Now we want to find some $t \in \mathbb{T}$ such that the vector $\tilde{v} = (v_1,\dots,v_{s-1}, t, v_{s+1},\dots,v_n)\in \mathbb{P}(\mathbb{T}^n)$,  is in $\tropbar(\mu)$. Then
		$v = \pr_{\{s\}}^{\trop}(\tilde{v}) \in \pr^{\trop}_{\{s\}}(\tropbar(\mu))$.
		
		If for every $C \in \mathcal{C}(\mu)$ the minimum in $\{ C_i + v_i \}_{i \in [n] \setminus s}$ is achieved at least twice, 
		we can set $t = \infty$ and we are done. Therefore, we  assume that there exists a circuit $C \in \mathcal{C}(\mu)$ such that the minimum in $\{ C_i + v_i \}_{i \in [n] \setminus s}$ is achieved only once. Let $t \in \mathbb{R}$ such that $t + C_s = \min_{i \in [n] \setminus s} \{ C_i + v_i \}$. Then, the minimum in $\{ C_i + \tilde{v}_i \}_{i \in [n]}$ is achieved twice. We claim that $\tilde{v} \in \tropbar(\mu)$. 
		
		We proceed by contradiction. Let $C' \in \mathcal{C}(\mu)$ and assume that the minimum in $\{ C'_i + \tilde{v}_i \}_{i \in [n]}$ is achieved only once at the index $j \in [n]$. Up to tropical scalar multiplication we can assume that $C'_s = C_s \neq \infty$.  Suppose first that $j \neq s$. By construction, $v_i + C_i \geq t + C_s = t + C'_s > v_j + C'_j$ for every $i \in [n]$, in particular $C_j \neq C'_j$. On the other hand, we have $v_i + C'_i > v_j + C'_j$ for every $i \neq j$, therefore $v_i + \min(C_i,C'_i) > v_j + C'_j$ for every $i \neq j$. From \cite[Theorem 3.4]{murota2001circuitvaluation} there exists a vector $C''$ of $\mu$ such that $C''_s = \infty$, $C''_i \geq \min\{ C_i,C'_i \}$ for all $i \in [n]$ with equality whenever $C_i \neq C'_i$, in particular $C''_j = C'_j$. But now $\supp(C'') \subseteq [n] \setminus s$, so the minimum in $\{ C''_i + v_i \}_{i \in [n] \setminus s}$ has to be achieved at least twice, contradicting $v_i + C''_i \geq v_i + \min(C_i,C_i') > v_j + C'_j$ for every $i \neq j$.
		
		Now suppose that $j = s$. Let $k \in [n]$ be the index at which the minimum in $\{v_i + C_i\}_{i \in [n] \setminus s}$ is achieved. Then we have $v_k + C_k = t + C_s = t + C'_s < v_i + C'_i$ for every $i \neq s$, in particular $C_k < C'_k$. Now, applying \cite[Theorem 3.4]{murota2001circuitvaluation} again, we obtain a vector $C''$ such that $C''_s = \infty$, $C''_i \geq \min\{ C_i,C'_i \}$ for all $i \in [n]$ with equality whenever $C_i \neq C'_i$, in particular $C''_k = C_k$. Then, $v_k + C''_k = v_k + C_k < v_i + C'_i$ and further $v_k + C''_k = v_k + C_k < v_i + C_i$ for every $i \neq s,k$. This contradicts the fact that the minimum in $\{v_i + C''_i\}_{i \in [n] \setminus s}$ is achieved at least twice. 
	\end{proof}
	
	\begin{theorem}\label{thm:ld_tropical_flags_are_ld_matroids}
		Let $\mu$ and $\nu$ be valuated matroids on a common ground set $[n]$. The following statements are equivalent
		\begin{enumerate}[label=(\arabic*)]
			\item $\mu_S \twoheadleftarrow \nu$,
			\item $\pr^{\trop}_{S}(\tropbar(\mu)) \subseteq \tropbar(\nu)$.
		\end{enumerate}
	\end{theorem}
	\begin{proof}
		From Theorem \ref{thm: brandt equivalence for tropical flags} $(b) \Leftrightarrow (c)$ and Proposition \ref{lem:projection of tropical linear space}, we have $\mu_S \twoheadleftarrow \nu$ if and only if $\pr_S^{\trop}(\tropbar(\mu)) = \tropbar(\mu_S) \subseteq \tropbar(\nu)$.
	\end{proof}
	
	\subsection{Morphisms of valuated matroids}
	
	In this section, we outline how we can recast the definition of linear degenerate valuated flag matroids in terms of morphisms of valuated matroids (as defined in \cite[Remark 4.3.3]{brandt2021tropicalflag}) and prove  (b) $\Leftrightarrow$ (d) of Theorem \ref{thm:main_valuated_ld_flag_correspondence} and Theorem \ref{thm:main_valuated_ld_flag_correspondence_realizable}. The advantage of doing so is that this allows for further generalizations. For instance one can define a more general \defemph{quiver Dressian} by using morphisms of valuated matroids in place of linear maps between linear spaces.
	
	Let $M$ (or $\mu$) be a (valuated) matroid on the ground set $[m]$. Let $o$ be an element not in $[m]$. We denote by $M_o$ (or $\mu_o$) the matroid $M \oplus U_{0,1}$ (or $\mu \oplus U_{0,1}$) obtained by adding $o$ as a loop. In addition, let $N$ be a matroid on the ground set $[n]$. 
	
	A \defemph{morphism} (or \defemph{strong map}) of matroids $f:M \rightarrow N$ is a map of sets $f:[m] \cup \{o\} \rightarrow [n] \cup \{o\}$ such that $f(o) = o$ and the inverse image of a flat in $N_o$ is a flat in $M_o$. Morphisms of matroids can be characterized in terms of matroid quotients.
	
	\begin{definition}
		Let $f:[m] \rightarrow [n]$ be a map of sets and $N$ a matroid over $[n]$. The \defemph{induced} matroid $f^{-1}(N)$ on $[m]$ is defined by $\rk_{f^{-1}(N)}(A) = \rk_N(f(A))$ for every $A \subseteq [m]$.
	\end{definition}
	
	\begin{lemma}[{\cite[Lemma 2.4]{eur2020morphisms}}]
		The map of sets $f:[m] \cup \{o\} \rightarrow [n] \cup \{o\}$ with $f(o) = o$ is a morphism of matroids if and only if $f^{-1}(N_o) \twoheadleftarrow M_o$.
	\end{lemma}
	
	Now we use the above characterization to extend the definition of morphism to valuated matroids, using quotients of valuated matroids.

	\begin{proposition}
		Let $f:[m] \rightarrow [n]$ be a surjective map of sets and let $\nu$ be a valuated matroid on $[n]$ of rank $r$ with underlying matroid $N$. The function $f^{-1}(\nu) : \binom{[m]}{r} \rightarrow \mathbb{T}$ defined by
		\[ f^{-1}(\nu)(I) = \begin{cases}
			\nu(f(I)) & \text{if $I$ is a basis of } f^{-1}(N), \\
			\infty & \text{otherwise,}
		\end{cases} \]
		is a valuated matroid with underlying matroid $f^{-1}(N)$.
	\end{proposition}
	\begin{proof}
		Since $f$ is surjective, $N$ and $f^{-1}(N)$ have the same rank $r$. Now if $B \in \binom{[m]}{r}$, then from $\rk_{f^{-1}(N)}(B) = \rk_N(f(B))$ we have that $B$ is a basis of $f^{-1}(N)$ if and only if $f(B)$ is a basis of $N$. This proves that the map $f^{-1}(\nu)$ is well defined. In particular, if $B$ is a basis of $f^{-1}(N)$, then $|B| = |f(B)|$, so the restriction of $f$ on $B$ is bijective.
		
		Now it suffices to show that for $I,J \in \binom{[m]}{r}$ and $i \in I \setminus J$ there exists $j \in J \setminus I$ such that
		\begin{equation}\label{eq: inequality}
			f^{-1}(\nu)(I) + f^{-1}(\nu)(J) \geq f^{-1}(\nu)(I \cup j \setminus i) + f^{-1}(\nu)(J \cup i \setminus j).
		\end{equation}
		If $I$ or $J$ is not a basis of $f^{-1}(N)$, then the left hand side of the above inequality is $\infty$ and we are done. Otherwise, assume that $I$ and $J$ are bases of $f^{-1}(N)$. This means that $f(I)$ and $f(J)$ are bases of $\nu$. Now, if $f(i) \in f(I) \setminus f(J)$, then there exists $f(j) \in$ $f(J)\setminus f(I)$ for some $j\in J\setminus I$ such that
		\begin{equation}\label{eq: morphism matroids inequality}
			\nu(f(I)) + \nu(f(J)) \geq \nu(f(I) \cup f(j) \setminus f(i)) + \nu(f(J) \cup f(i) \setminus f(j)).
		\end{equation}
		If $f(i) \in f(J)$, then there exists $j \in J\setminus I$ such that $f(i) = f(j)$. Therefore, we obtain equality in \eqref{eq: morphism matroids inequality}. Thus, for both cases, we get the inequality \eqref{eq: morphism matroids inequality}, which, using the definition of $f^{-1}(\nu)$ and \eqref{eq: inequality}, is exactly the definition of a valuated matroid.
	\end{proof}
	
	Let $\mu$ be a valuated matroid on the ground set $[m]$, and let $S \subseteq [m]$. Then, the \emph{restriction} of $\mu$ to $S$, denoted by $\mu_{|S}$, is defined as the deletion $\mu \setminus ([m] \setminus S)$ (see Proposition-Definition \ref{propdef: deletion of valuated matroid}).
	
	\begin{definition}
		Let $f:[m] \rightarrow [n]$ be a map of sets and let $\nu$ be a valuated matroid on the ground set $[n]$. The \defemph{induced} valuated matroid is defined by $f^{-1}(\nu) = f^{-1}(\nu_{|f([m])})$.
	\end{definition}
	
	\begin{definition}\label{def: morphism of matroid}
		Let $\mu$ and $\nu$ be valuated matroids on the ground set $[m]$ and $[n]$ respectively. A map of sets $f: [m] \cup \{o\} \rightarrow [n] \cup \{o\}$ with $f(o) = o$ is a \defemph{morphism of valuated matroids}, denoted $f: \mu \rightarrow \nu$, if $f^{-1}(\nu_o) \twoheadleftarrow \mu_o$. Further, $f$ is \emph{realizable} if the quotient $f^{-1}(\nu_o) \twoheadleftarrow \mu_o$ is realizable.
	\end{definition}
	
	In our context, it suffices to restrict to projections. Let $\mu$ and $\nu$ be two valuated matroids on the common ground set $[n]$, and let $S \subseteq [n]$.  Define the projection map $\pr_S: [n] \cup \{o\} \rightarrow [n] \cup \{o\}$ by
	\[ \pr_S(x) = \begin{cases}
		x & \text{if } x \notin S, \\
		o & \text{if } x \in S.
	\end{cases} \]
	Using this definition, we prove the final parts of Theorem \ref{thm:main_valuated_ld_flag_correspondence} and Theorem \ref{thm:main_valuated_ld_flag_correspondence_realizable}, (b) $\Leftrightarrow$ (d).
	\begin{proposition}\label{prop: b equi d}
		Let $\mu$ and $\nu$ be two valuated matroids on the common ground set $[n]$, and let $S \subseteq [n]$. Then, $\pr_S: \nu \rightarrow \mu$ is a morphism of valuated matroids if and only if $\mu_S \twoheadleftarrow \nu$. Further, $\pr_S: \nu \rightarrow \mu$ is realizable if and only if $\mu_S \twoheadleftarrow \nu$ is realizable.
	\end{proposition}
	\begin{proof}
		We define $\mu_S$ as in \eqref{eq: def of mu_S}. By construction, we have $\pr_S^{-1}(\mu_o) = {(\mu_o)}_S$. Thus $\pr_S: \nu \rightarrow \mu$ is a morphism of valuated matroids if and only if $(\mu_o)_S \twoheadleftarrow \nu_o$. Since being a quotient is a condition on the valuations of bases, and $\mu$ and $\mu_o$ (and $\nu$ and $\nu_o$ respectively) have the same bases with the same valuation, $(\mu_o)_S \twoheadleftarrow \nu_o$ if and only if $\mu_S \twoheadleftarrow \nu$. The realizability statement follows from Definition \ref{def: morphism of matroid}.
	\end{proof}
	
	
	\section{The poset of linear degenerate flag varieties}\label{sec: applications}
	
	Linear degenerate flag varieties can be arranged in a poset in a natural way as follows. Fix $k,n \in \mathbb{N}$ and a sequence $\mathbf{r} = ( r_1,\dots,r_k )$ of nonnegative integers such that $r_1 \leq \dots \leq r_k \leq n$. Now, we consider the following set of linear degenerate flag varieties:
	\[ \mathcal{L} = \bigg\{ \LFl(\mathbf{r},\mathbf{S};n) : \mathbf{S} = (S_1,\dots,S_{k-1}) \text{ with } S_i \subseteq [n] \bigg\}. \]
	We can define an order relation $\preceq$ on $\mathcal{L}$ given by $\LFl(\mathbf{r},\mathbf{S};n) \preceq \LFl(\mathbf{r},\mathbf{S}';n)$ if and only if $S_i \subseteq S_i'$ for every $i \in \{ 1,\dots,k \}$, where $\mathbf{S}' = ( S_1',\dots,S_{k-1}')$. Note that $(\mathcal{L},\preceq)$ is a finite lattice isomorphic to the product of lattices $\prod_{i=1}^k 2^{[n]}$, where $2^{[n]}$ is the power set of $[n]$ ordered by set inclusion.
	
	The maximum of $\mathcal{L}$ is the linear degenerate flag variety with $S_i = [n]$ for every $i$, in other words at each step of the flag we are projecting all the coordinates, so the linear spaces of the flag do not have any relation to each other, and the variety we obtain is just a product of Grassmannians:
	\[ \LFl(\mathbf{r},([n],\dots,[n]);n) = G(r_1;n) \times \dots \times G(r_k;n). \]
	On the other hand, the minimum of $\mathcal{L}$ is the linear degenerate flag variety with $S_i = \emptyset$ for every $i$.  Here, we are not degenerating the flag variety as at each step the projection is an identity map, thus all linear degenerate flags are flags:
	\[ \LFl(\mathbf{r},(\emptyset,\dots,\emptyset);n) = \Fl(\mathbf{r};n). \]
	
	Analogously, we can arrange linear degenerate tropical flag varieties, and linear degenerate flag Dressians in lattices isomorphic to $\prod_{i=1}^k 2^{[n]}$. 
	
	\paragraph{Linear degenerate tropical flag varieties with $n=4$}
	Now, we want to take a closer look at the lattice of linear degenerate tropical flag varieties for the case $n=4$.
	
	We used Macaulay2 \cite{M2} to compute the linear degenerate Pl\"ucker relations, using code available in the Github repository \cite{borzi2023githubrepo}; and used the package \verb*|Tropical.m2| \cite{tropicalpackage} to compute the respective linear degenerate tropical flag varieties. We did some additional computations in \texttt{gfan} \cite{gfan} and Oscar \cite{Oscar}.
	
	For the rest of this section, we will consider varieties of \emph{complete flags} in $\mathbb{C}^4$. More precisely, we fix $\mathbf{r} = (1,2,3)$, and omit $\mathbf{r}$ in our notation. For instance, we denote $\Fl(4)\coloneqq\Fl(\mathbf{r};4)$ and $\LFl \big((\{1\},\emptyset);4 \big)\coloneqq\LFl \big( \mathbf{r},(\{1\},\emptyset);4 \big)$. To simplify the notation, we will use $\LFl(S_1,S_2;4)$ in place of $\LFl((S_1,S_2);4)$. We consider tropicalization with respect to the trivial valuation.

	A point $p$ in the tropicalization of (linear degenerate) flag varieties corresponds to an arrangement of a point $v_p$, a general tropical line $L_p$ and a general tropical plane $P_p$ in ${\mathbb{P}(\mathbb{T}^4)}$. As in previous examples, $L_p$ has two vertices connected by a bounded edge, and four unbounded edges in the coordinate directions, two adjacent to each vertex. A general tropical plane in ${\mathbb{P}(\mathbb{T}^4)}$ consists of a vertex, four adjacent rays in all four coordinate directions, and all two-dimensional cones spanned by pairs of these rays (see Figure \ref{fig: tropical flag}).
	
	\begin{example}\label{ex: tropfl4}
		We begin by analyzing the tropicalization $\trop(\Fl(4))$ of the complete flag variety  $\Fl(4)$. By \cite[Theorem 5.2.1]{brandt2021tropicalflag}, $\trop(\Fl(4)) = \FlDr(4)$, i.e.,  the Pl\"ucker relations form a tropical basis and all tropical flags of length $4$ are realizable. 
		
		The tropical variety $\trop(\Fl(4))$ is a six-dimensional simplicial fan with lineality dimension three in ${\mathbb{P}(\mathbb{T}^4)}\times {\mathbb{P}(\mathbb{T}^6)} \times {\mathbb{P}(\mathbb{T}^4)}$. It has f-vector $(1,20,79,78)$ after quotienting by the lineality space. This variety was computed in \cite[Theorem 4]{bossinger2017toricdegenerations} and a sketch of it was given in \cite[Figure 9]{jell2022moduli}, which we report here in Figure \ref{fig: tropfl4}.
		
		A point $p$ in $\trop(\Fl(4))$ corresponds to a complete tropical flag, i.e. $v_p \subseteq L_p \subseteq P_p \subseteq {\mathbb{P}(\mathbb{T}^4)}$. An example of a complete flag was given in Figure \ref{fig: tropical flag}(a).
		
		After quotienting by the lineality space, $\trop(\Fl(4))$ can be seen as a ``tropical line bundle" over the Petersen graph (as explained in \cite[Paragraph 3.3.3]{jell2022moduli}). There are 15 combinatorially distinct ways to arrange a generic tropical line inside a fixed generic tropical plane in ${\mathbb{P}(\mathbb{T}^4)}$, corresponding to the edges of the Petersen graph in Figure \ref{fig: tropfl4}. In fact, this is dual to a fixed point inside a generic tropical line in ${\mathbb{P}(\mathbb{T}^4)}$, see Example \ref{ex: Fl 1,2;4}.
		
		The three blue edges in Figure \ref{fig: tropfl4}(a), connecting $(ab)$ to $(cd)$ for distinct $a,b,c,d \in [4]$ correspond to the case when one vertex of the tropical line $L_p$ lies on the ray of the tropical plane $P_p$ in direction $e_a+e_b$, and the other vertex of $L_p$ lies on the ray of $P_p$ in direction $e_c+e_d$. This situation is depicted in Figure \ref{fig: tropical flag}(a), where the dashed green rays are those in direction $e_a+e_b$. Further, the twelve black edges of the Petersen graph in Figure \ref{fig: tropfl4}(a)  connecting $(a)$ to $(ab)$ for distinct $a,b\in[4]$, correspond to $p$ where one vertex of $L_p$ lies on the ray of $P_p$ in direction $e_a$ and the other vertex of $L_p$ lies in the cone of $P_p$ spanned by $e_a$ and $e_b$.
		Finally, the the remaining degrees of freedom for the position of the point $p$ in the line bundle determine the position of the point $v_p$ on $L_p$.
		
		Note that both of the cases discussed above correspond to maximal cells: in the first case, both vertices of $L_p$ vary in the one-dimensional space in direction $e_a+e_b = -(e_c+e_d)$, whereas in the second case, the vertex of $L_p$ inside the span of $e_a$ and $e_b$ can be freely chosen, but then fixes the choice of the other vertex on $e_a$ by the balancing condition. The polyhedral structure of the tropical line bundle differs accordingly, see Figure \ref{fig: tropfl4}$(b)$ for blue edges and Figure \ref{fig: tropfl4}$(b')$ for black edges respectively.
	\end{example}
	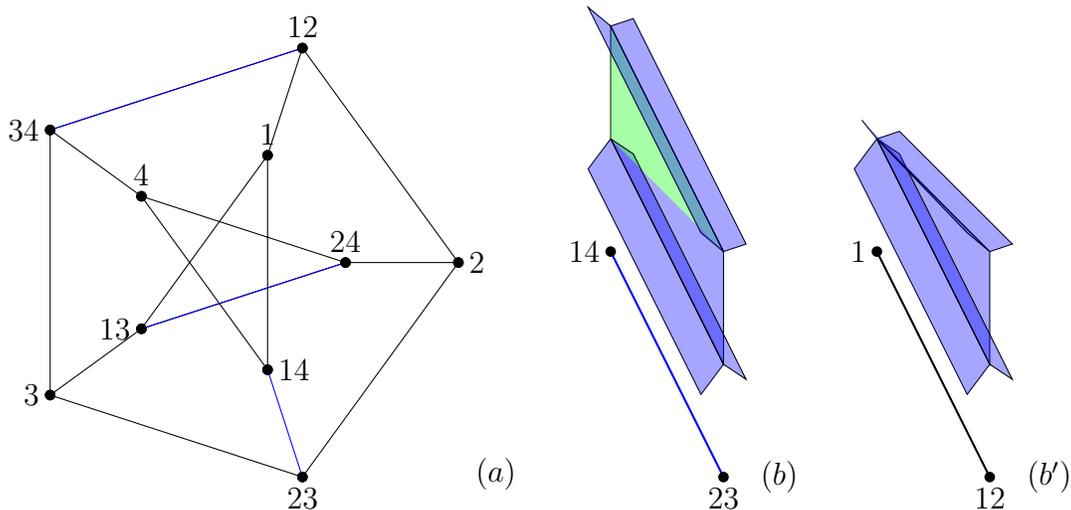
\begin{figure}[ht]
		\centering\begin{tabular}{c c c}
			\begin{tikzpicture}
				\foreach \i in {1,...,5}
				\fill (\i*360/5:1.5) coordinate (5\i) circle(2 pt);
				
				\foreach \i in {1,...,5}
				\fill (\i*360/5:3) coordinate (5\i) circle(2 pt);
				\draw  (360/5:3) --  (2*360/5:3) --  (3*360/5:3) --  (4*360/5:3) --  (5*360/5:3)-- (360/5:3);
				\node[above] at (360/5:1.5) {$1$};
				\node[above] at (2*360/5:1.5) {$4$};
				\node[above, left] at (3*360/5:1.5) {$13$};
				\node[above, right] at (4*360/5:1.5) {$14$};
				\node[above] at (5*360/5:1.5) {$24$};
				\node[above] at (360/5:3) {$12$};
				\node[above, left] at (2*360/5:3) {$34$};
				\node[above, left] at (3*360/5:3) {$3$};
				\node[below] at (4*360/5:3) {$23$};
				\node[above,right] at (5*360/5:3) {$2$};
				\draw  (360/5:3) -- (360/5:1.5);
				\draw  (2*360/5:3) -- (2*360/5:1.5);
				\draw  (3*360/5:3) -- (3*360/5:1.5);
				\draw[blue]  (4*360/5:3) -- (4*360/5:1.5);
				\draw  (5*360/5:3) -- (5*360/5:1.5);
				\draw  (3*360/5:1.5) -- (5*360/5:1.5) -- (2*360/5:1.5)-- (4*360/5:1.5)-- (1*360/5:1.5)-- (3*360/5:1.5);
				\draw[blue]  (3*360/5:1.5) -- (5*360/5:1.5);
				\draw[blue]  (1*360/5:3) -- (2*360/5:3);
				\foreach \i in {1,...,5}
				\fill (\i*360/5:1.5) coordinate (5\i) circle(2 pt);
				
				\foreach \i in {1,...,5}
				\fill (\i*360/5:3) coordinate (5\i) circle(2 pt);
				{\node[below] at (3.5,-2.45) {$(a)$};}
			\end{tikzpicture}	& \begin{tikzpicture}
				\draw[blue, thick] (0, 2) -- (1.5, -1);
				\fill (1.5, -1) circle (2pt);
				\fill (0,2) circle (2pt);
				\node[below] at (1.5,-1) {$23$};
				\node[left] at (0,2) {$14$};
				
				\draw (0, 3.5) -- (1.5, 0.5) --  (1.5,2) -- (0,5) --(0, 3.5) ;
				\fill[opacity = 0.35, blue](0, 3.5) -- (1.5, 0.5) --  (1.5,2);
				
				\fill[opacity = 0.35, blue](0, 3.5) -- (1.5, 0.5) --(1.8, 0.3) -- (0.3,3.3) -- (0, 3.5) ;
				\fill[opacity = 0.35, green](1.5,2) -- (0,5) --(0, 3.5);
				\draw (0, 3.5) -- (1.5, 0.5) -- (1.2, 0.1) -- (-0.3,3.1) -- (0, 3.5) ;
				\fill[opacity = 0.35, blue](0, 3.5) -- (1.5, 0.5) -- (1.2, 0.1) -- (-0.3,3.1) -- (0, 3.5) ;
				\draw (0, 3.5) -- (1.5, 0.5) --(1.8, 0.3) -- (0.3,3.3) -- (0, 3.5) ;
				\draw (1.5,2) -- (0,5) -- (-0.3,5.25) -- (1.2,2.25) -- (1.5,2);
				\fill[opacity = 0.35, blue] (1.5,2) -- (0,5) -- (-0.3,5.25) -- (1.2,2.25) -- (1.5,2); 
				\draw (1.5,2) -- (0,5) -- (0.3,5.1) -- (1.8,2.1) -- (1.5,2);
				\fill[opacity = 0.35, blue]  (1.5,2) -- (0,5) -- (0.3,5.1) -- (1.8,2.1) -- (1.5,2); 
				{\node [right] at (1.85,-1) {$(b)$};}\end{tikzpicture}  & \begin{tikzpicture}
				\draw[thick] (0, 2) -- (1.5, -1);
				\fill (1.5, -1) circle (2pt);
				\fill (0,2) circle (2pt);
				\node[below] at (1.5,-1) {$12$};
				\node[left] at (0,2) {$1$};
				
				\draw (0, 3.5) -- (1.5, 0.5) --  (1.5,2)  --(0, 3.5) ;
				
				\fill[opacity = 0.35, blue](0, 3.5) -- (1.5, 0.5) --  (1.5,2);
				
				\fill[opacity = 0.35, blue](0, 3.5) -- (1.5, 0.5) --(1.8, 0.3) -- (0.3,3.3) -- (0, 3.5) ;
				
				\draw (0, 3.5) -- (1.5, 0.5) -- (1.2, 0.1) -- (-0.3,3.1) -- (0, 3.5) ;
				
				\fill[opacity = 0.35, blue](0, 3.5) -- (1.5, 0.5) -- (1.2, 0.1) -- (-0.3,3.1) -- (0, 3.5) ;

				\draw (0, 3.5) -- (1.5, 0.5) --(1.8, 0.3) -- (0.3,3.3) -- (0, 3.5) ;

				\draw  (-0.2, 3.75) -- (0, 3.5) -- (1.2,2.25) -- (1.5,2);
				\fill[opacity = 0.35, blue]  (-0.2, 3.75) -- (0, 3.5) -- (1.2,2.25) -- (1.5,2); 
				\draw (1.5,2) --  (0,3.5) -- (0.3,3.6)-- (1.8,2.1) -- (1.5,2);
				\fill[opacity = 0.35, blue]  (1.5,2) -- (0,3.5) -- (0.3,3.6) -- (1.8,2.1) -- (1.5,2); 
				{\node [right] at (1.85,-1) {$(b')$};}
			\end{tikzpicture} 
		\end{tabular}
		\caption{ The tropical flag variety $\trop(\Fl_4)$ after quotienting by its lineality space, interpreted as a ``tropical line bundle" over the Petersen graph, see also \cite[Paragraph 3.3.3 and Figure 9]{jell2022moduli}.}
		\label{fig: tropfl4}
	\end{figure}
	\begin{example}\label{ex: lindeg 1 fl4}
		Now, we want to consider the linear degenerate tropical flag varieties $\trop(\LFl(\{1\},\emptyset;4))$ and $\trop(\LFl(\emptyset,\{1\};4))$. The tropical variety $\trop(\LFl(\{1\},\emptyset;4))$ parametrizes tropical flags such that $\pr_{1}(v_p) \subseteq L_p \subseteq P_p$. This is an extension of Example \ref{ex: linear degenerate section 3} by additionally requiring containment in a plane. An example of a point in $\trop(\LFl(\{1\},\emptyset;4))$ is depicted in Figure \ref{fig: tropical flag} (b). 
		
		The linear degenerate tropical flag variety $\trop(\LFl(\{1\},\emptyset;4))$is a six-dimensional simplicial fan in ${\mathbb{P}(\mathbb{T}^4)}\times{\mathbb{P}(\mathbb{T}^6)}\times {\mathbb{P}(\mathbb{T}^4)}$ with lineality dimension four. It has a familiar combinatorial structure: after quotienting by the lineality space, we obtain a fan over the Petersen graph. Again, the fifteen cones over the Petersen graph correspond to the fifteen combinatorially distinct ways of arranging a general line in a general plane (see Example \ref{ex: tropfl4}). This time, the position of the point $v_p$ imposes no additional combinatorial constraints, as by $\pr_1(v_p) \subseteq L_p$, $v_p$ is contained in the span of the $e_1$-ray of $L_p$.

		Comparing to Example \ref{ex: tropfl4}, the situation is much easier: the fan is a ``usual line bundle" (as opposed to a ``tropical line bundle" in Example \ref{ex: tropfl4}) over the Petersen graph.  In Figure \ref{fig: tropfl4_lindeg1}, we depict this degeneration.
		
		We can understand $\trop(\LFl(\emptyset,\{1\};4))$ similarly. The two different types of linear degenerate  flags described above are dual to each other, in fact $\LFl(\{1\},\emptyset;4) \simeq \LFl(\emptyset,\{1\};4)$, and   $\trop(\LFl(\{1\},\emptyset;4))$ and $\trop(\LFl(\emptyset,\{1\};4))$ have the same polyhedral structure with different lineality spaces.
	\end{example}
	
	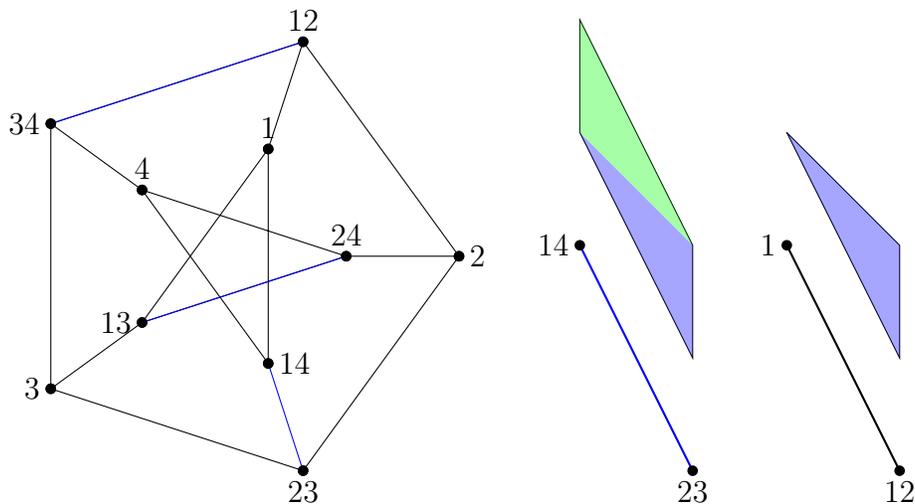
\begin{figure}[ht]
		\centering\begin{tabular}{c c c}
			\begin{tikzpicture}
				\foreach \i in {1,...,5}
				\fill (\i*360/5:1.5) coordinate (5\i) circle(2 pt);
				
				\foreach \i in {1,...,5}
				\fill (\i*360/5:3) coordinate (5\i) circle(2 pt);
				\draw  (360/5:3) --  (2*360/5:3) --  (3*360/5:3) --  (4*360/5:3) --  (5*360/5:3)-- (360/5:3);
				\node[above] at (360/5:1.5) {$1$};
				\node[above] at (2*360/5:1.5) {$4$};
				\node[above, left] at (3*360/5:1.5) {$13$};
				\node[above, right] at (4*360/5:1.5) {$14$};
				\node[above] at (5*360/5:1.5) {$24$};
				\node[above] at (360/5:3) {$12$};
				\node[above, left] at (2*360/5:3) {$34$};
				\node[above, left] at (3*360/5:3) {$3$};
				\node[below] at (4*360/5:3) {$23$};
				\node[above,right] at (5*360/5:3) {$2$};
				\draw  (360/5:3) -- (360/5:1.5);
				\draw  (2*360/5:3) -- (2*360/5:1.5);
				\draw  (3*360/5:3) -- (3*360/5:1.5);
				\draw[blue]  (4*360/5:3) -- (4*360/5:1.5);
				\draw  (5*360/5:3) -- (5*360/5:1.5);
				\draw  (3*360/5:1.5) -- (5*360/5:1.5) -- (2*360/5:1.5)-- (4*360/5:1.5)-- (1*360/5:1.5)-- (3*360/5:1.5);
				\draw[blue]  (3*360/5:1.5) -- (5*360/5:1.5);
				\draw[blue]  (1*360/5:3) -- (2*360/5:3);
				\foreach \i in {1,...,5}
				\fill (\i*360/5:1.5) coordinate (5\i) circle(2 pt);
				
				\foreach \i in {1,...,5}
				\fill (\i*360/5:3) coordinate (5\i) circle(2 pt);
			\end{tikzpicture}	& \begin{tikzpicture}
				\draw[blue, thick] (0, 2) -- (1.5, -1);
				\fill (1.5, -1) circle (2pt);
				\fill (0,2) circle (2pt);
				\node[below] at (1.5,-1) {$23$};
				\node[left] at (0,2) {$14$};
				
				\draw (0, 3.5) -- (1.5, 0.5) --  (1.5,2) -- (0,5) --(0, 3.5) ;
				
				\fill[opacity = 0.35, blue](0, 3.5) -- (1.5, 0.5) --  (1.5,2);
				
				\fill[opacity = 0.35, green](1.5,2) -- (0,5) --(0, 3.5);	\end{tikzpicture}  & \begin{tikzpicture}
				\draw[thick] (0, 2) -- (1.5, -1);
				\fill (1.5, -1) circle (2pt);
				\fill (0,2) circle (2pt);
				\node[below] at (1.5,-1) {$12$};
				\node[left] at (0,2) {$1$};
				
				\draw (0, 3.5) -- (1.5, 0.5) --  (1.5,2)  --(0, 3.5) ;
				
				\fill[opacity = 0.35, blue](0, 3.5) -- (1.5, 0.5) --  (1.5,2);

			\end{tikzpicture} 
		\end{tabular}
		\caption{ The linear degenerate tropical flag variety $\trop(\LFl(\emptyset,\{1\});4)$ can be interpreted as a ``line bundle" over the Petersen graph.}
		\label{fig: tropfl4_lindeg1}
	\end{figure}
	\begin{example}
		To finish, we consider the linear degenerate tropical flag variety $\trop(\LFl(\{1\},\{1\};4))$. It is a six-dimensional simplicial fan with lineality dimension five in ${\mathbb{P}(\mathbb{T}^4)}\times{\mathbb{P}(\mathbb{T}^6)}\times {\mathbb{P}(\mathbb{T}^4)}$ and f-vector $(1,3)$ after quotienting out the lineality space, i.e. a six-dimensional ``line bundle" over a tropical line. We depict it in Figure \ref{fig: tropfl4_lindeg2}. 
		Here, maximal cells can be interpreted as follows. As before, the first linear degeneration condition implies that $v_p$ lies on the span of the ray of $L_p$ in direction $e_1$. The second degeneration implies that one vertex of the tropical line lies on the linear span of $e_1$ (and thus that the other vertex lies on one of the 3 coordinate half-planes spanned by $e_1$). Note that this does not imply that $L_p$ is contained in $P_p$. Thus, there are three maximal cells.
	\end{example}
	
	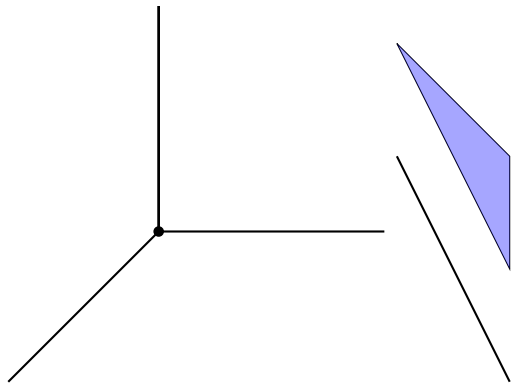
\begin{figure}[ht]
		\centering
		\begin{tabular}{c c}
			\begin{tikzpicture}[scale = 2]
				\draw[thick] (-1,-1) -- (0,0) -- (0,1.5) -- (0,0) -- (1.5,0);
				\fill (0,0) circle (1pt);
			\end{tikzpicture} 
			\begin{tikzpicture}
				\draw[thick] (0, 2) -- (1.5, -1);
				\draw (0, 3.5) -- (1.5, 0.5) --  (1.5,2)  --(0, 3.5) ;
				
				\fill[opacity = 0.35, blue](0, 3.5) -- (1.5, 0.5) --  (1.5,2);

			\end{tikzpicture} 
		\end{tabular}
		\caption{ The linear degenerate tropical flag variety $\trop(\LFl(\{1\},\{1\});4)$ can be interpreted as a ``line bundle" over the tropical line.}
		\label{fig: tropfl4_lindeg2}
	\end{figure}
	
	One possible application of the poset of linear degenerate tropical flag varieties would be to reduce the problem of computing a tropical flag variety to the problem of computing (a product of) tropical Grassmannians. Recall that the tropical flag variety is the minimum of the poset $\mathcal{L}$ of linear degenerate tropical flag varieties, while its maximum is a product of tropical Grassmannians. Therefore, one could try to start from the top of $\mathcal{L}$, and, by descending the poset $\mathcal{L}$ step by step, reconstruct the structure of the tropical flag variety. In order to do that, it would be enough to understand what happens at the \emph{covers} of the poset $\mathcal{L}$, that is, to (fully or partially) reconstruct the structure of $\trop(\LFl(\mathbf{r},\mathbf{S};n))$ from another linear degenerate tropical flag variety $\trop(\LFl(\mathbf{r},\mathbf{S}';n))$ that covers it, i.e.  $\mathbf{S}$ is obtained from $\mathbf{S}'$ by adding one element in one of the sets $S_i$.
	
	\begin{question}
		Can we reconstruct the structure of $\trop(\LFl(\mathbf{r},\mathbf{S};n))$ from a cover?
	\end{question}
	
	The examples we have seen above already provide some insight into what the answer is for complete flags with $n=4$. A common behaviour that we observe is that the lineality space increases in dimension after each linear degeneration. In fact, this can be shown in more generality. The next result shows that for an ideal $I \subseteq k[x_0,\dots,x_n]$, the lineality space of a tropical variety $\trop(V(I))$ contains the homogeneity space of $I$, which is the linear subspace of vectors $v \in \mathbb{R}^{n+1}$ such that $I$ is homogeneous with respect to the grading $\deg(x_i) = v_i$.
	
	\begin{lemma}\label{lem:homogeneous lineality space}
		Let $I \subseteq k[x_0,\dots,x_n]$ be an ideal, where $k$ is a field with the trivial valuation. Let $v = (v_0,\dots,v_n) \in \mathbb{R}^{n+1}$. If $I$ is homogeneous with respect to the grading $\deg(x_i) = v_i$ then $v$ is in the lineality space of $\trop(V(I))$.
	\end{lemma}
	\begin{proof}
		If $I$ is homogeneous with respect to the grading $\deg(x_i) = v_i$, then $\iin_v(f) = f$ for every $f \in I$. This implies that $\iin_{v+w}(f) = \iin_w(f)$, as for every monomial $m$ of $f$, we are adding the same weight to the scalar product of the exponent vector of $m$ and $w$. In particular $\iin_{w+v}(I) = \iin_{w}(I)$ for every $w \in \mathbb{R}^{{n+1}}$. Hence $w \in \trop(V(I))$ if and only if $w + v \in \trop(V(I))$, that is, $v$ is in the lineality space of $\trop(V(I))$.
	\end{proof}

	
	
	\begin{corollary}
		The homogeneity space of a linear degenerate tropical flag variety is contained in the homogeneity space of every linear degenerate tropical flag variety in $\mathcal{L}$ that covers it.
	\end{corollary}
	\begin{proof}
		The claim follows from the structure of the linear degenerate Pl\"{u}cker relations. In fact, for a fixed grading of the Pl\"{u}cker variables, if the polynomials in $\mathscr{P}_{r,s;S;n}$ are homogeneous with respect to this grading, then polynomials in $\mathscr{P}_{r,s;S';n}$ are also homogeneous with respect to this grading for every $S' \supseteq S$.
	\end{proof}

	By looking at the examples in the previous section, one might be tempted to conjecture that a cover relation on the poset implies set inclusion on the tropical varieties. In general, this is false, as the following example shows.
	
	\begin{example}
		In this example, we are going to show that
		\[ \trop \big( \LFl((1,2),\emptyset;4) \big) \nsubseteq \trop \big( \LFl((1,2),\{ 1 \};4) \big). \]
		We already described the above tropical varieties in Example \ref{ex: Fl 1,2;4} and Example \ref{ex: linear degenerate section 3}. Now, assume that our base field $K$ is the \emph{field of Laurent series} $\mathbb{K}((t))$, that is the  quotient field of the DVR $\mathbb{K}[[t]]$ of formal power series with coefficients in a field $\mathbb{K}$ in the variable $t$. Then, $K$ has valuation $v:K \rightarrow \mathbb{T}$ where $v(f(t))$ is the minimum of the exponents appearing in $f$.
		
		Now let $a,b \in \mathbb{Q}$ with $b > a > 0$, and consider the two matrices
		\[
		A_1 = \begin{pmatrix} 1 & 1 & 1 & 1 \end{pmatrix}, \quad A_2 = \begin{pmatrix} 1 & 1 & 1 &1 \\ t^a & 0 & t^b & 1 \end{pmatrix}.
		\]
		Let $L_1, L_2 \subseteq \mathbb{K}^4$ be the two linear spaces generated by the rows of the matrices $A_1$ and $A_2$ respectively. By construction, $L_1 \subseteq L_2$, but $\pr_{\{1\}}(L_1) \nsubseteq L_2$. We can see this through the Pl\"{u}cker equations computed in Example \ref{ex: linear degenerate section 3} as follows. The valuations of the Pl\"{u}cker coordinates of $L_1$ and $L_2$ are
		\begin{gather*}
			(p_1,p_2,p_3,p_4) = (0,0,0,0), \\ (p_{1,2},p_{1,3},p_{1,4},p_{2,3},p_{2,4},p_{3,4}) = (a,a,0,b,0,0).
		\end{gather*}
		In particular, the minima in all the tropical polynomials in $\mathscr{P}_{2;4}^{\trop}$ and $\mathscr{P}_{1,2;4}^{\trop}$ are achieved at least twice, while the minimum in, for instance, the second tropical polynomial of $\mathscr{P}_{1,2;\{1\};4}^{\trop}$ listed in Example \ref{ex: linear degenerate section 3}, $p_{4}p_{1,2} \oplus p_{2}p_{1,4}$ is not achieved twice:
		\[ p_{4}p_{1,2} = 0 \odot a = a > 0 = 0 \odot 0 = p_{2}p_{1,4}.\qedhere \]
	\end{example}
	
	While we do not obtain containment on tropical flag varieties or Dressians in the poset of linear degenerations, from the definition of the linear degenerate Pl\"{u}cker relations, we obtain the following containment on some boundary components.
	
	\begin{corollary}
		Let $\LFlDr(r,r',S \cup \{ s \},n) \prec \LFlDr(r,r',S,n)$ be a cover in the poset of linear degenerate flag Dressians. Set
		\[ \mathscr{B} = \left\{ (p_I) \in \mathbb{T}^{\binom{n}{r}} \times \mathbb{T}^{\binom{n}{r'}} : p_I = \infty \text{ for every } I \in \binom{[n]}{r} \text{ such that } s \in I \right\}. \]
		Then we have
		\[ \LFlDr(r,r',S,n) \cap \mathscr{B} \subseteq \LFlDr(r,r',S \cup \{ s \},n). \]
	\end{corollary}

	Another interesting application of the poset of linear degenerate flag varieties concerns \emph{relative realizability}. We say that two realizable tropical linear spaces $T_1\subseteq T_2$ are \defemph{relatively realizable} if there exist realizations $L_1$ of $T_1$ and $L_2$ of $T_2$ such that $L_1 \subseteq L_2$. Let $\mathcal{L}$ be the poset of linear degenerate tropical flag varieties with flags of length $2$ in $\mathbb{P}(\mathbb{T}^n)$ and rank vector $(r,s)$. Then, accurately describing the cover relations of $\mathcal{L}$  provides us a way to solve the relative realizability problem: the maximal element of $\mathcal{L}$ is $\trop(G(r;n))\times \trop(G(s;n))$ in which we impose no conditions on either containment or relative realizability, whereas the minimal element $\trop(\Fl(r,s;n))$ of $\mathcal{L}$ does. Thus, if we could explicitly reconstruct $\trop(\Fl(r,s;n))$ from $\trop(G(r;n))\times \trop(G(s;n))$, we would have an explicit solution to the relative realizability problem by tracking elements in the cover relations.
	
	\bibliographystyle{plain}
	\bibliography{Reference.bib}
	
	\vspace{0.25cm}
	\noindent
	\textsc{Alessio Borz\`i, MPI MiS Leipzig, Inselstraße 22, 04103 Leipzig, Germany}\\
	\textit{Email:} alessio.borzi@mis.mpg.de 
	\vspace{0.25cm}
	
	\noindent
	\textsc{Victoria Schleis, Universit\"at T\"ubingen, Auf der Morgenstelle 10,
		72076 T\"ubingen,
		Germany} \\
	\textit{Email :} victoria.schleis@student.uni-tuebingen.de
\end{document}